\documentclass[11pt]{amsart}
\usepackage{amsmath}
\usepackage{amsfonts}
\usepackage{amssymb}
\usepackage{amsthm}
\usepackage{mathtools} 

\usepackage{wrapfig}
\usepackage{caption}

\usepackage{subfigure}
\usepackage{float}

\theoremstyle{plain}
\newtheorem{theorem}{Theorem}[section]
\newtheorem{proposition}[theorem]{Proposition}

\theoremstyle{definition}
\newtheorem{definition}[theorem]{Definition}

\theoremstyle{remark}
\newtheorem{remark}[theorem]{Remark}

\numberwithin{equation}{section} 
\numberwithin{figure}{section}   

\usepackage[square,comma,numbers,sort]{natbib}
\usepackage[colorlinks=true, pdfborder={ 0 0 0}]{hyperref}
\hypersetup{urlcolor=blue, citecolor=red}
\usepackage{url}

\newcommand{\field}[1]{\mathbb{#1}}

\newcommand{\nZ}{\field{Z}}

\newcommand{\nR}{\field{R}}

\newcommand{\nT}{\mathbb T}
\newcommand{\vphi}{\varphi}

\newcommand{\tac}{\textasteriskcentered}

\newcommand{\pnt}[1]{\left(#1\right)}

\DeclareMathOperator*{\esssup}{ess\,sup}

\newcounter{my_counter}
\title[Regularity criteria for KSE]
{Regularity criteria for the Kuramoto-Sivashinsky equation in dimensions two and three}

\date{}

\author{Adam Larios*}
\address[Adam Larios*]{Department of Mathematics, 
                University of Nebraska--Lincoln,
        Lincoln, NE 68588-0130, USA}
\email[Adam Larios*]{alarios@unl.edu}
\author{Mohammad Mahabubur Rahman}
\address[Mohammad Mahabubur Rahman]{Department of Mathematics and Statistics, 
                Texas Tech University, 
       Lubbock, TX, 79409, USA}
\email[Mohammad Mahabubur Rahman]{Mohammad-Mahabu.Rahman@ttu.edu}
\author[Kazuo Yamazaki]{Kazuo Yamazaki}
\address[Kazuo Yamazaki]{Department of Mathematics and Statistics, 
                Texas Tech University, 
       Lubbock, TX, 79409, USA}
\email[Kazuo Yamazaki]{kyamazak@ttu.edu}

\keywords{(Kuramoto-Sivashinsky, Navier-Stokes Equations, Regularity, Global Well-Posedness.)}
\thanks{MSC 2010 Classification: 35A01, 35K25, 35K51, 35K58, 35B65, 35B10, 65M70.\\\tac Corresponding author}

\allowdisplaybreaks

\usepackage{placeins} 

\begin{document}
\begin{abstract}
We propose and prove several regularity  criteria for the 2D and 3D Kuramoto-Sivashinsky equation, in both its scalar and vector forms.  In particular, we examine integrability criteria for the regularity of solutions in terms of the scalar solution $\phi$, the vector solution $u\triangleq\nabla\phi$, as well as the divergence $\text{div}(u)=\Delta\phi$, and each component of $u$ and $\nabla u$.  We also investigate these criteria computationally in the 2D case, and we include snapshots of solutions for several quantities of interest that arise in energy estimates.
\end{abstract}

\maketitle
\thispagestyle{empty}

\noindent
\section{Introduction}\label{secInt}
\noindent
For those with an interest in nonlinear PDEs, the two-dimensional (2D) Kuramoto-Sivashinsky equation (KSE) is full of tantalizing possibilities.
There is such strong diffusion in the equation that the nonlinear term seems incapable of overpowering it, and yet, it is precisely the structure of the nonlinear term that has so far thwarted all efforts at a general proof of global well-posedness.  It is sometimes thought that this is due to the backward diffusion term in the equations, but this term serves mainly as the source of energy and instability, fueling a kind of self-perpetuating chaos.  Instead, the major difficulty is due to the fact that the nonlinear term does not vanish in any known energy estimates, similar to the vorticity stretching term in the Navier-Stokes equations (NSE), where an analogous barrier arises.  This paper aims to quantify these difficulties by proving several integrability-type component based regularity criteria for well-posedness, and also to shed light on well-posedness issues via numerical simulations.  
 
 For those with a more practical eye, the KSE enjoys a wealth of applications.  Originally proposed in the late 1970's by Kuramoto, Sivashinsky, and Tsuzuki in investigations of crystal growth \cite{Kuramoto_Tsuzuki_1975,Kuramoto_Tsuzuki_1976} and flame-front instabilities \cite{Sivashinsky_1977} (see also \cite{Sivashinsky_1980_stoichiometry}), it has since made many appearances, such as in studies of inclined planes \cite{sivashinsky1980vertical}, and has even been shown to be a general feature of certain unstable behaviors \cite{misbah1994secondary}.  Although we consider both 2D and 3D cases, we note that, at least in certain situations such as flame-front propagation, the 2D case is more relevant, as it describes the evolving 2D front of a 3D flame.  However, the 3D case is of interest in terms of making analogies with the 3D NSE.
 
 For the $N$D case, short-time existence of smooth solutions (specifically, Gevrey class regularity) was proved in \cite{Biswas_Swanson_2007_KSE_Rn}. The triviality of steady states was studied in \cite{Cao_Titi_2006_KSE}.  Variations of the 2D KSE have been studied in, e.g., \cite{Ambrose_Mazzucato_2021,Boling_Fengqiu_1993_JPDE,CotiZelati_Dolce_Feng_Mazzucato_2021,Feng_Mazzucato_2020,Ioakim_Smyrlis_2016,Larios_Yamazaki_2020_rKSE,Tomlin_Kalogirou_Papageorgiou_2018}, as well as variations on its boundary conditions \cite{Galaktionov_Mitidieri_Pokhozhaev_2008,Larios_Titi_2015_BC_Blowup,Pokhozhaev_2008}.   The question of the global well-posedness of KSE for $N\geq 2$ in the periodic case or $\nR^N$ is still open in general; however, in dimensions $N=2$ and $3$ for the case of radially symmetric initial data in an annular domain, global well-posedness was proved in \cite{Bellout_Benachour_Titi_2003}, assuming homogeneous Neumann boundary conditions.  
It was shown in \cite{Ambrose_Mazzucato_2018} (see also \cite{Ambrose_Mazzucato_2021,CotiZelati_Dolce_Feng_Mazzucato_2021,Feng_Mazzucato_2020}) that, in the case where there are no linearly growing modes, global existence holds for sufficiently small initial data in a certain function space based on the Wiener algebra. We also mention \cite{Sell_Taboada_1992}, which studied global existence and attractors in 2D thin domains (see also \cite{Boling_Fengqiu_1993_JPDE,Benachour_Kukavica_Rusin_Ziane_2014_JDDE_2DKSE}).  

The two major difficulties in the $N\geq2$ case are (i) the fact that the nonlinear term does not vanish in energy estimates (since the solution is not divergence-free), and hence no $L^p$ norm is conserved, and (ii) the fact that, due to the 4th-order derivatives, no maximum principle is known to hold.  However, we note an interesting recent result in \cite{Ibdah_2021_MichelsonSivashinsky} that proves global regularity (without any smallness condition) for a modified version of the Michelson-Sivashinsky equation.  While lower-order than KSE, this equation shares many similarities with KSE in that it also does not conserve $L^p$ norms, has no known maximum principle, and has no divergence-free condition.

In contrast to the 2D KSE, the 1D KSE is a seemingly limitless playground, where one has essentially everything one could want.  Well-posedness is straight-forward \cite{Nicolaenko_Scheurer_1984,Tadmor_1986}, the large-time dynamics are chaotic  (unlike in the case of the 1D Burgers' equation, where the large-time dynamics are trivial) but finite-dimensional \cite{Collet_Eckmann_Epstein_Stubbe_1993_Attractor,Nicolaenko_Scheurer_Temam_1985,Otto_2009}, and much work on quantities of interest, such as the existence of the attractor and estimates on its dimension, have seen excellent progress in recent decades, see, e.g. \cite{Collet_Eckmann_Epstein_Stubbe_1993_Attractor,Collet_Eckmann_Epstein_Stubbe_1993_Analyticity,Constantin_Foias_Nicolaenko_Temam_1989,Constantin_Foias_Nicolaenko_Temam_1989_IM_Book,Foias_Nicolaenko_Sell_Temam_1985,Foias_Sell_Temam_1985,Foias_Sell_Titi_1989,Goluskin_Fantuzzi_2019,Goodman_1994,Grujic_2000_KSE,Hyman_Nicolaenko_1986,Ilyashenko_1992,Nicolaenko_Scheurer_Temam_1986,Robinson_2001,Tadmor_1986,Temam_1997_IDDS,Kostianko_Titi_Zelik_2018,Wittenberg_2014_DCDSA} and the references therein.
 
Faced with an obstacle in proving global well-posedness in case $N = 2$ or 3, a natural strategy is to investigate its criterion following analogous works for the 3D NSE such as \eqref{estimate 10} from  \cite{Escauriaza_Seregin_Sverak_2003, Kiselev_Ladyzhenskaya_1957, Prodi_1959, Serrin_1962}. In this endeavor, we first obtained criterion that seemed to be natural extensions from those of the NSE (see \eqref{1.12} and Remark \ref{Remark 1.4}). To our surprise, subsequently we were able to improve such results significantly  (see Theorem \ref{Theorem 2.1} and Remark \ref{Remark 3.2}). This motivated us to pursue another direction of research that has caught much attention in the past few decades, specifically component reduction of such classical regularity criterion (see \eqref{estimate 11}). Despite a large amount of work dedicated to such results on various systems including the NSE, magnetohydrodynamics (MHD) system and surface quasi-geostrophic (SQG) equations, divergence-free property of velocity field was crucial in all of their results; consequently, we are not aware of any example of a PDE that does not involve divergence-free velocity field and yet admit component reduction results. Unexpectedly, we were able to obtain such results by making use of special structure of the KSE, which seems to be a unique property that is absent in the NSE or Burgers' equation (see Remark \ref{Remark 3.5}). 

Now let us write $\partial_{t} \triangleq \frac{\partial}{\partial t}$ and introduce several equations of our main concern. First, the $N$D KSE in vector form is given by
\begin{align}
\label{KSE}
\partial_{t} u+(u\cdot\nabla)u + \lambda \Delta u + \Delta^{2}u = 0 ,
\end{align}
with appropriate boundary conditions (here taken to be periodic) and initial data $u^{in}$.  Here, $\lambda>0$ is a constant. 
We note that one often considers the KSE with $\lambda=1$ and a domain with size determined by a parameter $L>0$, such as $[-\tfrac{L}{2},\tfrac{L}{2}]^N$. We choose to work with a fixed domain $\mathbb{T}^N\triangleq[-\pi,\pi]^N$ and a general parameter $\lambda$; these formulations are equivalent under the rescaling 
$u'(x,t)
\mapsto 
(L/2\pi)^{-3}u(\tfrac{x}{L/2\pi},\tfrac{t}{(L/2\pi)^4})$ after setting $\lambda=(L/2\pi)^2$.  

The scalar form of \eqref{KSE} is formally given by
\begin{align}
\label{KSE_scalar}
\partial_{t} \phi +\tfrac{1}{2}|\nabla\phi|^2 + \lambda \Delta\phi + \Delta^2\phi &= 0,
\end{align}
By setting $u=\nabla\phi$ in \eqref{KSE_scalar}, one formally obtains \eqref{KSE}.  Our analytical results throughout this manuscript go through independently of the fact that $u = \nabla \phi$; i.e., we will obtain results on \eqref{KSE} and \eqref{KSE_scalar} separately and independently of each other. For comparison, let us recall the NSE:
\begin{equation}\label{NSE}
\partial_{t} u + (u\cdot\nabla) u + \nabla \Pi = \Delta u, \hspace{3mm} \nabla\cdot u = 0, 
\end{equation} 
where $\Pi$ represents the pressure. Due to the condition $\nabla\cdot u = 0$, an $L^{2}(\mathbb{T}^{N})$-inner product of $(u\cdot\nabla) u$ with $u$ vanishes for the NSE, which is crucial in energy estimates.  However, this does not happen for the KSE, nor for the $N$D Burgers' equation
\begin{equation}\label{Burgers}
\partial_{t} u + (u\cdot\nabla) u = \Delta u
\end{equation}
with $N \geq 2$, although Burgers' equation enjoys a maximum principle allowing for a proof of global well-posedness; see, e.g., \cite{Ladyzhenskaya_1968,Pooley_Robinson_2016} for details.  Currently, no such maximum principle is known to hold for the KSE, and hence, global well-posedness for arbitrary smooth initial data remains an open problem.
The case $\Pi \equiv 0$ and $\nabla\cdot u = 0$ condition being dropped reduces the NSE to Burgers' equation.

\begin{remark}\label{u_neq_nabla_phi}
A major interest of ours is the Navier-Stokes equations \eqref{NSE}, where, not counting the pressure, the nonlinearity is (at least notationally) the same as in equation \eqref{KSE}. Hence, our we aim to prove as much as possible {\it without} using the assumption that $u=\nabla\phi$.  However, all of our results can be translated to results about equation \eqref{KSE_scalar} by thinking of $u$ as merely a notation for $\nabla\phi$; for instance, $u_1 = \partial_{1}\phi$, $\nabla\cdot u = \Delta u$, and so on.  

Note that if one does \textit{not} assume that $u=\nabla\phi$, then the vector and scalar forms of the KSE might not be equivalent.  For instance, if $\phi$ is a given \textit{smooth} solution to \eqref{KSE_scalar} with initial data $\phi^{in}$, then $u\triangleq\nabla\phi$ automatically solves \eqref{KSE} with initial data $u^{in}\triangleq\nabla\phi^{in}$.  However, given a solution $u$ to the vector form \eqref{KSE}, it may not be possible to find a solution $\phi$ to the scalar form such that $u=\nabla\phi$, since $u$ might not be a pure gradient; for instance, when the initial data $u_0$ of \eqref{KSE} is not a pure gradient.
 \end{remark} 
 
 \vspace{-2mm}
 

 %
%
\begin{wrapfigure}[13]{r}{0.5\textwidth}
\vspace{-4mm}
\centering
\captionsetup{width=0.5\textwidth}
\includegraphics[width=0.5\textwidth,trim=4mm 3mm 19mm 9mm, clip]
{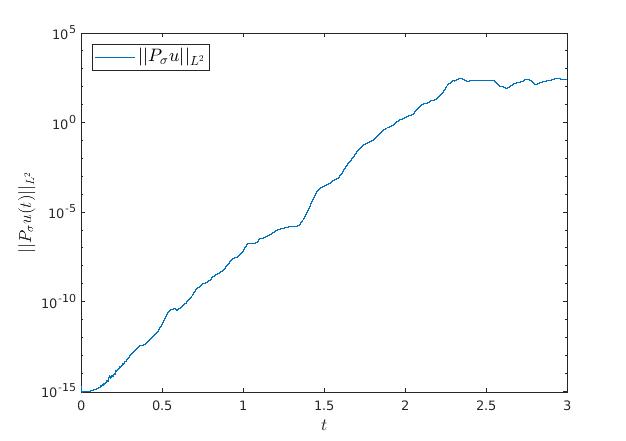}
\vspace{-6mm}
\caption{\label{fig_Pu} \scriptsize \textit{(Log-linear plot)} The growth of the $L^2$-norm of the divergence-free part $v=P_\sigma u$ of a vector solution $u$ of \eqref{KSE} ($\lambda=8.1$; i.e.,  24 unstable modes), with initial data $u^{in}=\nabla\phi^{in}$, where $\phi^{in}$ is given by \eqref{KKP}.  Since $P_\sigma u^{in}\equiv0$, one  expects that $P_\sigma u(t)\approx0$ for $t>0$, but this is not what is observed computationally. }
\end{wrapfigure}
\noindent
 Indeed, denoting the Helmholtz-Hodge decomposition of a solution $u$ to \eqref{KSE} by $u=\nabla q + v$ where $\nabla\cdot v=0$, a straight-forward calculation shows, at least formally, that
 \begin{minipage}[t]{0.45\textwidth}
  \vspace{-2mm}
 \begin{align}
 &
\frac12\frac{d}{dt}\|v\|_{L^2}^2
+\|\Delta v\|_{L^2}^2
\\=&\notag
\lambda\|\nabla v\|_{L^2}^2
-2(\nabla q\cdot S,v),
\end{align}
 \vspace{-0.75mm}
\end{minipage}

\noindent
where $S\triangleq\tfrac12(\nabla v - (\nabla v)^T)$ is the anti-symmetric part of the gradient of $v$.  If $(\nabla q\cdot S,v)\not\equiv0$, there is potential for the gradient part of $u$ to ``push'' the solution off of the gradient manifold.  Moreover, even if one could show analytically\footnote{We denote the Leray-Helmholtz projection $P_\sigma u = v$.  See, e.g., \cite{Constantin_Foias_1988,Temam_2001_Th_Num} for details.} that $P_\sigma u = 0$ is preserved by \eqref{KSE}, in computations, small errors could be amplified due to the destabilizing term in \eqref{u_neq_nabla_phi}, and the non-gradient part of the solution could grow rapidly.  To test this computationally, we simulated the vector equation \eqref{KSE} starting from a pure-gradient initial condition.  The results are shown in Figure \ref{fig_Pu}, where we observe that the deviation from the gradient manifold grows exponentially fast in time (see Figure \ref{fig_Pu}).


Therefore, we choose to study both \eqref{KSE} and \eqref{KSE_scalar} independently, although if one assumes that $u = \nabla \phi$, then all the results we obtain for \eqref{KSE}, specifically Theorems \ref{Theorem 2.1}, \ref{Theorem 2.7}, \ref{new_theorem}, \ref{Theorem 2.9}, \ref{Theorem 3.10}, and \ref{Theorem 2.13}, immediately imply equivalent results for $\phi$ that solves \eqref{KSE_scalar}. Informally, we give a comprehensive list our results of regularity criteria here for clarity:
\begin{enumerate}
\item in terms of $u$ in 2D and 3D (Theorem \ref{Theorem 2.1});
\item in terms of $u_{1}$ in 2D (Theorem \ref{Theorem 2.7});
\item in terms of $u_{1}, u_{2}$ in 3D (Theorem \ref{new_theorem}); 
\item in terms of $\nabla u$ in 2D and 3D (Theorem \ref{Theorem 2.9}); 
\item in terms of $\nabla\cdot u$ in 2D and 3D (Theorem \ref{Theorem 3.10}); 
\item in terms of $\partial_{2} u_{2}$ in 2D (Theorem \ref{Theorem 2.13}); \item in terms of $\phi$ in 2D and 3D (Theorem \ref{Theorem 3.17});
\item in terms of $\partial_{12} \phi$ in 2D (Theorem \ref{Theorem 3.14}).
\end{enumerate}
\begin{remark}
We note that the proofs of the results on individual components  are non-standard, and do not follow as in, e.g., the Navier-Stokes case; a major difference being the loss of the divergence-free condition, which is what makes the multi-dimensional KSE difficult in the first place.
\end{remark}

This paper is organized as follows.  In Section \ref{sec_pre}, we lay out notation and mention some preliminary results.  In Section \ref{sec_reg}, we state and prove our main results. In Section \ref{sec_comp}, we present and discuss our computational results. 

\section{Preliminaries}\label{sec_pre}

\noindent
For simplicity we write $\partial_{j} \triangleq \frac{\partial}{\partial x_{j}}$ for $j \in \{1,\hdots, N\}$ and $\int F \triangleq \int_{\mathbb{T}^{N}} F(x) dx$; for further brevity we may write $f_{x} \triangleq \partial_{x} f$, etc., when $f$ is a scalar field. We write $A \lesssim_{a,b} B$ when there exists a constant $C = C(a,b) \geq 0$ such that $A \leq C B$.  We recall that we may write 
\begin{equation*} 
f(x) = \sum_{k \in \mathbb{Z}^{N}} \hat{f}(k) e^{ik\cdot x}, 
\end{equation*} 
and the inhomogeneous and homogeneous Sobolev norms 
\begin{equation*}
\lVert f \rVert_{H^{s}} \triangleq \Big(\sum_{k \in \mathbb{Z}^{N}}  (1+ \lvert k \rvert^{2s} ) \lvert \hat{f} (k)\rvert^{2} \Big)^{\frac{1}{2}}, \hspace{3mm} \lVert f \rVert_{\dot{H}^{s}} \triangleq \Big(\sum_{k \in \mathbb{Z}^{N}} \lvert k \rvert^{2s} \lvert \hat{f}(k) \rvert^{2}\Big)^{\frac{1}{2}},
\end{equation*} 
respectively. We denote by $(\cdot,\cdot)$ the standard $L^2(\mathbb{T}^{N})$ inner-product, and by $\left<\cdot,\cdot\right>$ the $H^{-1}$-$H^1$ duality paring.

Let us formally write down the definition of a strong solution to the KSE equation \eqref{KSE}.
\begin{definition}
We call $u$ a strong solution to the KSE \eqref{KSE} over a time interval $[0,T]$ if for any  $\psi \in C^{\infty}(\mathbb{T}^{N})$, 
\begin{equation}
\left<\partial_{t}u, \psi\right> + (u\cdot\nabla u, \psi) + (\Delta u, \Delta \psi) = \lambda (\nabla u, \nabla \psi) 
\end{equation} 
for almost all $t \in [0,T]$, and 
\begin{subequations}
\begin{align}
& u \in L^{\infty} ([0, T]; H^{1}(\mathbb{T}^{N})), u \in L^{2}([0, T]; H^{3}(\mathbb{T}^{N})),\\
& u \in C([0,T]; H^{s}(\mathbb{T}^{N})) \hspace{1mm} \forall \hspace{1mm} s \in [0,1), \partial_{t} u \in L^{2}([0, T]; H^{-1}(\mathbb{T}^{N})). 
\end{align}
\end{subequations} 
and moreover, the initial data $u_0\in H^1(\mathbb{T}^{N}))$ is satisfied in the sense that $(u(t),\psi)\rightarrow(u_0,\psi)$ as $t\downarrow0$.
\end{definition} 
Standard arguments show (as in, e.g., \cite{Larios_Yamazaki_2020_rKSE}) that the KSE is locally well-posed in $H^{1}(\mathbb{T}^{N})$ for both $N \in\{ 2, 3\}$: 
\begin{theorem}\label{Theorem 2.2}
In case $N \in \{2,3\}$, given any initial data $u^{in} \in H^{1}(\mathbb{T}^{N})$, there exists an interval $[0, T)$ where $T = T(u^{in}) > 0$ over which the KSE \eqref{KSE} has a unique strong solution starting from $u^{in}$. 
\end{theorem}  
While the proof is standard and follows from the works within \cite{Larios_Yamazaki_2020_rKSE}, because we could not locate its proof in the literature, especially the case $N =3$, for completeness we sketch its outline in the Appendix. On the other hand, by \cite[Theorem  1]{Bellout_Benachour_Titi_2003}, we know that \eqref{KSE_scalar} is locally well-posed in $L^{2}(\mathbb{T}^{N})$ for both $N \in \{2,3\}$. 

Let us recall two prominent regularity criteria for the 3D NSE \eqref{NSE}: 
\begin{subequations}\label{estimate 10}
\begin{align}
& u \in L_{T}^{r}L_{x}^{p} \text{ where } \frac{3}{p} + \frac{2}{r} \leq 1,  p \in [3,\infty], \label{Serrin criteria}\\
& \nabla u \in L_{T}^{r} L_{x}^{p} \text{ where } \frac{3}{p} + \frac{2}{r} \leq 2, p\in [\tfrac{9}{4}, \infty] \label{Beirao criteria}
\end{align} 
\end{subequations}
due to \cite{Beirao_1995, Iscauriaza_Seregin_Sverak_2003, Prodi_1959, Serrin_1962}. Let us follow the proof of \eqref{Serrin criteria} on NSE \eqref{NSE} to examine what type of criteria we may expect for the $N$D KSE. Taking $L^{2}(\mathbb{T}^{N})$-inner products on \eqref{KSE} with $-\Delta u$ gives us 
\begin{equation}\label{1.6}
\frac{1}{2} \frac{d}{dt} \lVert \nabla u \rVert_{L^{2}}^{2} + \lVert \Delta \nabla u \rVert_{L^{2}}^{2} = \int (u\cdot\nabla) u \cdot \Delta u + \lambda \lVert \Delta u \rVert_{L^{2}}^{2}. 
\end{equation} 
Let us already mention that in the case of the KSE \eqref{KSE}, we need to additionally estimate $L^{2}(\mathbb{T}^{N})$-norm so that the sum with \eqref{1.6} gives an $H^{1}(\mathbb{T}^{N})$-estimate because the solution to the KSE \eqref{KSE} does not preserve mean-zero property in sharp contrast to the NSE as $\int (u\cdot\nabla)u \neq 0$ (this is certainly not a problem if we assume $u = \nabla \phi$). As we will see in \eqref{star}-\eqref{estimate 7}, this will not give us much trouble, and because we only want to illustrate heuristics here, let us omit these details.  Now, for clarity, we consider $p \in (N, \infty)$ first, and the case $p = \infty$ afterward. We estimate by H$\ddot{\mathrm{o}}$lder's inequality
\begin{equation}\label{1.7}
\int (u\cdot\nabla) u \cdot \Delta u \leq \lVert u \rVert_{L^{p}} \lVert \nabla u \rVert_{L^{2}} \lVert \Delta u \rVert_{L^{\frac{2p}{p-2}}}.  
\end{equation} 
We can continue by Gagliardo-Nirenberg inequality of 
\begin{align}
\lVert \nabla f \rVert_{L^{\frac{2p}{p-2}}} \lesssim& \lVert f \rVert_{L^{2}}^{\frac{p-N}{2p}} \lVert f \rVert_{H^{2}}^{\frac{p+N}{2p}} \label{1.8} \\
& \text{ for } p\in 
\begin{cases}
(N, \infty]  & \text{ if } N = 2, \\
[N, \infty] & \text{ if } N = 3
\end{cases}   \nonumber 
\end{align}
to bound by 
\begin{align}
\int (u\cdot\nabla) u \cdot \Delta u \lesssim& \lVert u \rVert_{L^{p}} \lVert \nabla u \rVert_{L^{2}} \lVert \nabla u \rVert_{L^{2}}^{\frac{p-N}{2p}} \lVert  \nabla u \rVert_{H^{2}}^{\frac{p+N}{2p}} \nonumber \\
\leq& \frac{1}{4} \lVert \Delta \nabla u \rVert_{L^{2}}^{2} + C \Big(\lVert u \rVert_{L^{p}}^{\frac{4p}{3p-N}} + \lVert u \rVert_{L^{p}}\Big) \lVert \nabla u \rVert_{L^{2}}^{2} \label{1.9}
\end{align} 
by Young's inequality. This estimate is not valid for the case $N = p = 2$ because it is false that $\lVert \Delta u \rVert_{L^{\infty}}$ can be bounded by a constant multiple of $\lVert \Delta u \rVert_{H^{1}}$ when $N  = 2$. In case $p = + \infty$, we easily use the interpolation of 
\begin{align}
\lVert f \rVert_{\dot{H}^{2}}^{2} \leq \Big( \int \lvert \lvert \xi \rvert^{3} \lvert \hat{f} (\xi) \rvert \rvert^{2} d \xi\Big)^{\frac{1}{2}} \Big( \int \lvert \lvert \xi \rvert \lvert \hat{f} (\xi) \rvert \rvert^{2} d\xi \Big)^{\frac{1}{2}}  
= \lVert f \rVert_{\dot{H}^{3}} \lVert  f \rVert_{\dot{H}^{1}} \label{estimate 1}
\end{align} 
and estimate 
\begin{align}
\int (u\cdot\nabla) u \cdot \Delta u \leq& \lVert u \rVert_{L^{\infty}} \lVert \nabla u \rVert_{L^{2}} \lVert \Delta u \rVert_{L^{2}} \nonumber\\
\leq& \frac{1}{4} \lVert \Delta \nabla u \rVert_{L^{2}}^{2} + C \lVert u \rVert_{L^{\infty}}^{\frac{4}{3}} \lVert \nabla u \rVert_{L^{2}}^{2}\label{1.10}
\end{align} 
by Young's inequality. Using \eqref{estimate 1} we can also estimate the linear term 
\begin{equation}\label{1.11}
\lambda \lVert \Delta u \rVert_{L^{2}}^{2} \leq \lambda \lVert \Delta \nabla u \rVert_{L^{2}} \lVert \nabla u \rVert_{L^{2}} \leq \frac{1}{4} \lVert \Delta \nabla u \rVert_{L^{2}}^{2} + C \lVert \nabla u \rVert_{L^{2}}^{2}. 
\end{equation} 
Applying \eqref{1.9}-\eqref{1.11} to \eqref{1.6} leads to 
\begin{equation*}
\frac{1}{2} \frac{d}{dt} \lVert \nabla u \rVert_{L^{2}}^{2} + \frac{1}{2} \lVert \Delta \nabla u \rVert_{L^{2}}^{2}  \lesssim \lVert \nabla u \rVert_{L^{2}}^{2} ( \lVert u \rVert_{L^{p}}^{\frac{4p}{3p-N}} + \lVert u \rVert_{L^{p}} + 1) 
\end{equation*} 
with $\frac{4p}{3p-N} = \frac{4}{3}$ in case $p = \infty$. 
This leads to a regularity criteria of 
\begin{align}
u \in L_{T}^{r}L_{x}^{p} 
&\text{ where } 
\begin{cases}
p \in (N, \infty], r \in [\frac{4}{3}, 2) & \text{ if } N = 2, \\
p \in [N, \infty], r \in [\frac{4}{3}, 2] & \text{ if } N = 3, 
\end{cases} \nonumber \\
& \text{ satisfy } \frac{N}{p} + \frac{2}{r} = \frac{3}{2} + \frac{N}{2p}. \label{1.12} 
\end{align}
\begin{remark}\label{Remark 1.4}
If we take $N = 2$ and then $p > 2$ arbitrarily close to 2 in \eqref{1.12}, then we see that the right hand side of \eqref{1.12} is almost 2, very reminiscent of \eqref{Beirao criteria} in terms of $\nabla u$ for the 3D NSE \eqref{NSE}. In fact, \eqref{1.12} may be heuristically explained via scaling argument as follows. Informally assuming that $\lambda = 0$ in \eqref{KSE}, we realize that if $u(t,x)$ is its solution, then so is $u_{\beta} (t,x) \triangleq \beta^{3} u(\beta^{4} t, \beta x)$ and $\lVert u_{\beta} \rVert_{L_{T}^{r} L_{x}^{p}} = \lVert u \rVert_{L_{\beta^{4} T}^{r} L_{x}^{p}}$ if and only if $\frac{N}{p} + \frac{2}{r} = \frac{3}{2} + \frac{N}{2p}$. We are indebted to one of the referees for pointing this out. 
\end{remark} 

As we emphasized after \eqref{1.9}, we cannot obtain the case $N = p = 2$ due to the fact that $H^{1}(\mathbb{T}^{2})\hookrightarrow L^{\infty} (\mathbb{T}^{2})$ is false in general. For the NSE, there are typically div-curl lemma or Brezis-Wainger type inequality arguments to work around this issue and improve $p > N$ to $p = N$. Unfortunately, both methods faced technical difficulty for the KSE. In fact, we will actually see in Theorem \ref{Theorem 2.1} that it is possible to improve \eqref{1.12} significantly by rather an unconventional approach that seems unique to the KSE. 

Very recently, the research direction of component reduction from the criteria in \eqref{estimate 10} for the 3D NSE has received much attention. While we acknowledge that there have been many extensions and improvements, let us list the following examples as a reference: starting from an initial data that is sufficiently smooth, there exists a unique solution to the 3D NSE for all $t > 0$ if for any $j, k \in \{1,2,3\}$, 
\begin{subequations}\label{estimate 11} 
\begin{align}
& u_{j} \in L_{T}^{r}L_{x}^{p} \text{ where } \frac{3}{p} + \frac{2}{r} \leq \frac{5}{8}, \hspace{1mm} p \in [\tfrac{24}{5}, \infty] \label{estimate 12}\\
& \partial_{j} u_{j} \in L_{T}^{r}L_{x}^{p} \text{ where } \frac{3}{p} + \frac{2}{r} < \frac{ 3(p+2)}{4p}, \hspace{1mm} p > 2, \hspace{1mm} r \in [1, \infty), \label{estimate 13} \\
& \partial_{k} u_{j} \in L_{T}^{r} L_{x}^{p} \text{ for } k \neq j \text{ where } \frac{3}{p} + \frac{2}{r} < \frac{p+3}{2p},  \hspace{1mm} p > 3, \hspace{1mm} r \in [1, \infty), \label{estimate 14} 
\end{align}
\end{subequations}
due to  \cite{CT11, Kukavica_Ziane_2006} (see also \cite{Cao_Titi_2008, Kukavica_Ziane_2007_JMP, ZP10}). We point out that while the dependence on $u = (u_{1}, u_{2}, u_{3})$ or $\nabla u = (\partial_{j}u_{k})_{1\leq j,k\leq 3}$ was remarkably reduced down respectively to $u_{j}$, $\partial_{j}u_{j}$, or $\partial_{k}u_{j}$, the integrability condition of 1 and 2 in \eqref{estimate 10} deteriorated respectively to $\frac{5}{8}$,  $\frac{3(p+2)}{4p}$, or $\frac{p+3}{2p}$ in \eqref{estimate 11}. This is the price one must pay for reducing components; to the best of our knowledge, despite much effort by many mathematicians, it remains unknown if we can improve $\frac{5}{8}$ up to 1 or $\frac{3(p+2)}{4p}, \frac{p+3}{2p}$ up to 2, as in \eqref{estimate 10}. All the proofs of such component reduction results rely crucially on the divergence-free property of $u$ in \eqref{NSE}, and some of such results have been extended to other systems of PDEs that involve divergence-free velocity field (e.g., \cite{Cao_Wu_2010, Yamazaki_2016} on MHD system; \cite{Yamazaki_2013_SQG} on SQG equations). In short, the proof of \eqref{estimate 12} relies on the fact that upon $\lVert (\partial_{1}u, \partial_{2}u) \rVert_{L^{2}}^{2}$-estimate of the 3D NSE \eqref{NSE}, one can separate $u_{3}$ within the nonlinear terms as follows:
\begin{align*}
\int (u\cdot\nabla)u \cdot (\partial_{1}^{2} + \partial_{2}^{2}) u =& - \sum_{i,j=1}^{3}\sum_{k=1}^{2}\int \partial_{k} u_{i}\partial_{i}u_{j} \partial_{k}u_{j} \\
\lesssim& \int \lvert u_{3} \rvert \lvert \nabla u \rvert \lvert \nabla (\partial_{1},\partial_{2}) u \rvert.
\end{align*}
E.g., we can estimate by summing the terms when $(i,j,k) = (1,1,2)$ and $(i,j,k) = (2,1,2)$ to bound 
\begin{align*}
\int \partial_{2}u_{1}\partial_{1}u_{1}\partial_{2}u_{1} + \partial_{2}u_{2}\partial_{2}u_{1}\partial_{2}u_{1} =& - \int (\partial_{2}u_{1})^{2}\partial_{3}u_{3} \\
\lesssim&  \int \lvert u_{3} \rvert \lvert \nabla u \rvert \lvert \nabla (\partial_{1}, \partial_{2}) u  \rvert 
\end{align*} 
where the divergence-free property played an important role. Despite the large number of papers devoted to this direction of research, we are not aware of any example of a PDE that admitted such component reduction of regularity criteria without a divergence-free property. Remarkably, we have been able to prove such results for the KSE (see Theorems \ref{Theorem 2.7}, \ref{new_theorem}, \ref{Theorem 2.13}, and \ref{Theorem 3.14}).  

\section{\texorpdfstring{Regularity criteria for $N$d KSE and their proofs}{}}\label{sec_reg}
The following computations hold on $\mathbb{R}^{N}$ or $\mathbb{T}^{N}$ for $N \in \{2,3\}$, but we choose to focus on the latter, and also elaborate more in the case $N = 2$ as our main concern. Results in Subsections \ref{Subsection 3.1}-\ref{Subsection 3.4} focus on \eqref{KSE} independently of \eqref{KSE_scalar} while Subsection \ref{Subsection 3.5} focuses on \eqref{KSE_scalar} independently of \eqref{KSE}. 

\subsection{\texorpdfstring{Result on $u$}{}}\label{Subsection 3.1}
Let us continue the $\dot{H}^{1}(\mathbb{T}^{N})$-estimate from \eqref{1.6}. For the sake of generality, let us also aim to attain a criterion in $W_{x}^{m,p}$ for $m \in [0, 1)$ so that taking $m = 0$ reduces to the standard criterion in $L_{x}^{p}$. We proceed as follows: 
\begin{align}\label{2.0}
\int (u\cdot\nabla) u \cdot \Delta u 
\leq \lVert u \rVert_{L^{\frac{Np}{N- mp}}} \lVert \nabla u \rVert_{L^{\frac{2Np}{Np + mp - N}}} \lVert \Delta u \rVert_{L^{\frac{2Np}{Np + mp - N}}}
\end{align} 
by H$\ddot{\mathrm{o}}$lder's inequality for $p$ in an appropriate range to be specified subsequently. Now we use two Gagliardo-Nirenberg inequalities: 
\begin{subequations}
\begin{align}
 \lVert f \rVert_{L^{\frac{2N p}{N p + mp - N}}} \lesssim& \lVert f \rVert_{L^{2}}^{\frac{ (4+ m) p - N}{4p}} \lVert f \rVert_{H^{2}}^{\frac{N - mp}{4p}}   \label{GN1} \\
& \text{ for any } N \in \{2,3\}, p \in 
\begin{cases}
[1, \frac{N}{m}]  &\text{ if } m \in (0,1), \\
[1, \infty] & \text{ if } m = 0, \\
\end{cases} \nonumber \\
 \lVert \nabla f \rVert_{L^{\frac{2N p}{N p + mp - N}}} \lesssim& \lVert f \rVert_{L^{2}}^{\frac{ (2+ m) p -N}{4p}} \lVert f \rVert_{H^{2}}^{\frac{ (2-m) p + N}{4p}} \label{GN2}\\
& 
\begin{cases}
p \in [1, \infty] & \text{ if } N = 2, m \in (0,1), \\
p \in [\frac{3}{m+2}, \infty] & \text{ if } N = 3, m \in (0,1),\\
p \in (1, \infty] & \text{ if } N = 2, m = 0, \\
p \in [\frac{3}{2}, \infty] & \text{ if } N = 3, m=0.
\end{cases}
\nonumber 
\end{align}
\end{subequations}   
Now if $m \in (0,1)$, then we apply Sobolev embedding $W^{m,p}(\mathbb{T}^{N}) \hookrightarrow L^{\frac{Np}{N-mp}} (\mathbb{T}^{N})$ assuming $p < \frac{N}{m}$ as part of hypothesis in case $m \in (0,1)$ and see that 
\begin{align}
\int (u\cdot\nabla) u \cdot \Delta u \lesssim&  \lVert u \rVert_{W^{m,p}} \lVert \nabla u \rVert_{L^{2}}^{\frac{ (m+3) p - N}{2p}} \lVert \nabla u \rVert_{H^{2}}^{\frac{ (1-m) p + N}{2p}} \nonumber \\
\leq& \frac{1}{8} \lVert \Delta \nabla u \rVert_{L^{2}}^{2} + C( \lVert u \rVert_{W^{m,p}}^{\frac{ 4p}{(3+m) p - N}} + 1) \lVert \nabla u \rVert_{L^{2}}^{2} \label{estimate 6}
\end{align}  
by Young's inequality. As mentioned, we also need an $L^{2}(\mathbb{T}^{N})$ estimate; thus, taking $L^{2}(\mathbb{T}^{N})$ inner-products on \eqref{KSE} with $u$ we obtain 
\begin{equation}\label{star}
\frac{1}{2} \frac{d}{dt} \lVert u \rVert_{L^{2}}^{2}+  \lVert \Delta u \rVert_{L^{2}}^{2} = \lambda \lVert \nabla u \rVert_{L^{2}}^{2} - \int (u\cdot\nabla) u \cdot u 
\end{equation}
where we estimate 
\begin{align}
& \int (u\cdot\nabla ) u \cdot u  \nonumber \\
\leq& \lVert u \rVert_{L^{\frac{N p}{N - mp}}} \lVert u \rVert_{L^{\frac{2N p}{N p + mp - N}}} \lVert \nabla u \rVert_{L^{\frac{2N p}{N p + mp - N}}} \nonumber \\
\lesssim&  \lVert u \rVert_{W^{m,p}} \lVert u \rVert_{L^{2}}^{\frac{ (m+3) p - N}{2p}} \lVert u \rVert_{H^{2}}^{\frac{ (1-m) p + N}{2p}} \nonumber \\
\lesssim& \lVert u \rVert_{W^{m,p}} \Big( \lVert u \rVert_{L^{2}}^{\frac{ (m+3) p - N}{2p}} \lVert \nabla u \rVert_{L^{2}}^{\frac{ (1-m) p + N}{4p}} \lVert \Delta \nabla u \rVert_{L^{2}}^{\frac{ (1-m) p +N}{4p}} + \lVert u \rVert_{L^{2}}^{2}\Big) \nonumber \\
\leq& \frac{1}{8} \lVert \Delta \nabla u \rVert_{L^{2}}^{2} + C\Big( 1+ \lVert u \rVert_{W^{m,p}} + \lVert u \rVert_{W^{m,p}}^{\frac{ 4p}{(3+m) p - N}}\Big) \lVert u \rVert_{H^{1}}^{2}  \label{estimate 7}
\end{align} 
by H$\ddot{\mathrm{o}}$lder's inequality, Sobolev embedding $W^{m,p} (\mathbb{T}^{N}) \hookrightarrow L^{\frac{Np}{N-mp}}(\mathbb{T}^{N})$ assuming $p < \frac{N}{m}$ in case $m \in (0,1)$ as part of hypothesis, \eqref{GN1}-\eqref{GN2}, and Young's inequality. Applying \eqref{1.11} and  \eqref{estimate 6} to \eqref{1.6}, and \eqref{estimate 7} to \eqref{star}, we showed the following criteria:
\begin{theorem}\label{Theorem 2.1}
Suppose $u \in L_{T}^{r}W_{x}^{m,p}$ where $m \in [0, 1)$, and 
\begin{equation}\label{range1}
\begin{cases} 
p \in [1, \frac{N}{m}),  r \in (\frac{4}{3}, \frac{4}{1+m}] & \text{ if } N = 2, m \in (0, 1),\\
p \in [\frac{N}{m+2}, \frac{N}{m}),  r \in (\frac{4}{3}, 4]& \text{ if } N = 3,  m \in (0,1),
\end{cases}
\end{equation} 
and 
\begin{equation}\label{range2}
\begin{cases} 
p \in ( \frac{N}{2}, \infty],  r \in [\frac{4}{3}, 4) & \text{ if } N = 2, m =0,\\
p \in [\frac{N}{2}, \infty],  r \in [\frac{4}{3}, 4]& \text{ if } N = 3,  m = 0,
\end{cases}
\end{equation} 
satisfy 
\begin{align}
\frac{N}{p} + \frac{2}{r} = \frac{3+m}{2} + \frac{N}{2p}. \label{range3}
\end{align}
Then the $N$D KSE is globally well-posed in $H^{1} (\mathbb{T}^{N})$. 
\end{theorem} 
\begin{remark}\label{Remark 3.2}
In comparison to Remark \ref{Remark 1.4},  because we improved the range of $p$ in \eqref{1.12} to \eqref{range2} in Theorem \ref{Theorem 2.1}, the right hand side of \eqref{range3} with $m =0$ can achieve $\frac{3}{2}+ \frac{N}{2p} > 2$ for $p \in (\frac{N}{2}, N)$; in fact, taking $p \approx \frac{N}{2}$ gives $\frac{3}{2} + \frac{N}{2p}$ to be almost $\frac{5}{2}$. We recall from Remark \ref{Remark 1.4} that $\frac{N}{p} + \frac{2}{r} = \frac{3}{2} + \frac{N}{2p}$ can be explained heuristically via scaling argument. We also emphasize that this new approach in the proof of Theorem \ref{Theorem 2.1} is impossible for the NSE. Indeed, if we choose $\lVert \Delta u \rVert_{L^{\frac{2Np}{Np + mp - N}}}$ in the first H$\ddot{\mathrm{o}}$lder's inequality of \eqref{2.0}, then this term already cannot be handled by the diffusion in the NSE, specifically $\lVert \Delta u \rVert_{L^{2}}^{2}$ in $H^{1}$-estimate. It is the fourth-order diffusion in the KSE that is allowing us to proceed with this new approach.
\end{remark}

\subsection{\texorpdfstring{Result on $u_{1}$}{}}\label{Subsection 3.2}
Next, we obtain a criterion on only one component $u_{1}$ in the 2D case while on $(u_{1}, u_{2})$ in the 3D case.  

\begin{proposition}\label{Proposition 2.3}
Suppose $u_{1} \in L_{T}^{r} W_{x}^{m,p}$ where $m \in [0,1)$ and 
\begin{equation*}
\begin{cases}
p \in [1, \frac{2}{m}), r \in (\frac{4}{3}, \frac{4}{1+m}] & \text{ if } m \in (0,1), \\
p \in (1, \infty], r \in [\frac{4}{3}, 4) & \text{ if } m = 0, 
\end{cases} 
\end{equation*} 
satisfy 
\begin{align*}
\frac{2}{p} + \frac{2}{r} = \frac{1}{p} + \frac{3+m}{2}. 
\end{align*}
Then $u_{2} \in L_{T}^{\infty} L_{x}^{2} \cap L_{T}^{2} H_{x}^{2}$. 
\end{proposition}
\begin{proof}
Taking $L^{2}(\mathbb{T}^{2})$-inner products with $u_{2}$ on the second component of \eqref{KSE} gives 
\begin{equation}\label{2.2o}
\frac{1}{2} \frac{d}{dt} \lVert u_{2} \rVert_{L^{2}}^{2} + \lVert \Delta u_{2} \rVert_{L^{2}}^{2} = - \int (u\cdot\nabla) u_{2} u_{2} + \lambda \lVert \nabla u_{2} \rVert_{L^{2}}^{2}
\end{equation} 
where 
\begin{align}
\lambda \lVert \nabla u_{2} \rVert_{L^{2}}^{2} \leq \lambda \lVert u_{2} \rVert_{L^{2}} \lVert \Delta u_{2} \rVert_{L^{2}} \leq \frac{1}{4} \lVert \Delta u_{2} \rVert_{L^{2}}^{2} + C \lVert u_{2} \rVert_{L^{2}}^{2}, \label{estimate 2}
\end{align} 
and 
\begin{align}
- \int (u\cdot\nabla) u_{2} u_{2} =& -\int u_{1} (\partial_{1} u_{2}) u_{2} - \int u_{2} \partial_{2} u_{2} u_{2}\nonumber \\
=& - \int u_{1} (\partial_{1} u_{2}) u_{2} - \int \frac{1}{3} \partial_{2} (u_{2})^{3} \nonumber \\
\leq& \lVert u_{1} \rVert_{L^{\frac{2p}{2-mp}}} \lVert u_{2} \rVert_{L^{\frac{4p}{2p + mp-2}}} \lVert \nabla u_{2} \rVert_{L^{\frac{4p}{2p + mp - 2}}} \nonumber\\
\lesssim& \lVert u_{1} \rVert_{W^{m,p}} \lVert u_{2} \rVert_{L^{2}}^{\frac{ (3+m) p - 2}{2p}} \lVert u_{2} \rVert_{H^{2}}^{\frac{ (1-m) p + 2}{2p}} \nonumber \\
\leq& \frac{1}{4} \lVert \Delta u_{2} \rVert_{L^{2}}^{2} + C \Big( \lVert u_{1} \rVert_{W^{m,p}}^{\frac{ 4p}{(3+m) p - 2}} + \lVert u_{1} \rVert_{W^{m,p}} \Big) \lVert u_{2} \rVert_{L^{2}}^{2} \label{2.2a} 
\end{align} 
by H$\ddot{\mathrm{o}}$lder's inequality, Sobolev embedding $W^{m,p} (\mathbb{T}^{2}) \hookrightarrow L^{\frac{2p}{2-mp}}(\mathbb{T}^{2})$ if $m \in (0,1)$ so that $p < \frac{2}{m}$ by hypothesis, \eqref{GN1}-\eqref{GN2}, and Young's inequality. Applying \eqref{estimate 2}-\eqref{2.2a} to \eqref{2.2o} implies the desired result. 
\end{proof} 

\begin{proposition}\label{Proposition 2.4}
Suppose $u_{1} \in L_{T}^{r} W_{x}^{m,p}$ where $m \in [0,1)$ and 
\begin{equation*}
\begin{cases}
p \in [1, \frac{2}{m}), r \in (\frac{4}{3}, \frac{4}{1+m}] & \text{ if } m \in (0,1), \\
p \in (1, \infty], r \in [\frac{4}{3}, 4) & \text{ if } m = 0, 
\end{cases} 
\end{equation*} 
satisfy 
\begin{align*}
\frac{2}{p} + \frac{2}{r} = \frac{1}{p} + \frac{3+m}{2}. 
\end{align*}
Then $u_{1} \in L_{T}^{\infty} L_{x}^{2} \cap L_{T}^{2} H_{x}^{2}$. 
\end{proposition} 

\begin{proof}
We take $L^{2}(\mathbb{T}^{2})$-inner products on the first component of $\eqref{KSE}$ with $u_{1}$ and compute  
\begin{align}\label{2.2b}
\frac{1}{2} \frac{d}{dt} \lVert u_{1} \rVert_{L^{2}}^{2} + \lVert \Delta u_{1} \rVert_{L^{2}}^{2} = - \int (u\cdot\nabla) u_{1} u_{1} + \lambda \lVert \nabla u_{1} \rVert_{L^{2}}^{2} 
\end{align} 
where $\lambda \lVert \nabla u_{1} \rVert_{L^{2}}^{2} \leq \frac{1}{4} \lVert \Delta u_{1} \rVert_{L^{2}}^{2} + C \lVert u_{1} \rVert_{L^{2}}^{2}$ similarly to \eqref{estimate 2}, while 
\begin{align}\label{estimate 5}
- \int (u\cdot\nabla) u_{1} u_{1} =& - \int u_{2} \partial_{2} u_{1} u_{1} \nonumber\\
\leq& \lVert u_{2} \rVert_{L^{\infty}} \lVert \nabla u_{1} \rVert_{L^{2}} \lVert u_{1} \rVert_{L^{2}} \nonumber\\
\lesssim& \lVert u_{2} \rVert_{H^{2}} \lVert u_{1} \rVert_{L^{2}}^{\frac{3}{2}} \lVert \Delta u_{1} \rVert_{L^{2}}^{\frac{1}{2}}
\leq \frac{1}{4} \lVert \Delta u_{1} \rVert_{L^{2}}^{2} + C \lVert u_{2} \rVert_{H^{2}}^{\frac{4}{3}} \lVert u_{1} \rVert_{L^{2}}^{2} 
\end{align} 
by Sobolev embedding $H^{2}(\mathbb{T}^{2}) \hookrightarrow L^{\infty}(\mathbb{T}^{2})$, \eqref{estimate 2} and Young's inequality. This implies that $u_{1} \in L_{T}^{\infty} L_{x}^{2} \cap L_{T}^{2}H_{x}^{2}$ because we know $\lVert u_{2}(t) \rVert_{H^{2}} \in L_{T}^{2}$ due to Proposition \ref{Proposition 2.3}
\end{proof} 

\begin{remark}\label{Remark 3.5}
Interestingly, the previous proof of Proposition \ref{Proposition 2.4} would not work for the 2D Burgers' equation if we neglect the fact that its solution has the maximum principle in the following way. Starting from $u_{1} \in L_{T}^{r}L_{x}^{p}$ for some $p$ and $r$, we can still deduce that $u_{2} \in L_{T}^{\infty} L_{x}^{2} \cap L_{T}^{2}H_{x}^{1}$ analogously to Proposition \ref{Proposition 2.3}. However, this does not seem to be sufficient to lead to $u_{1} \in L_{T}^{\infty} L_{x}^{2} \cap L_{T}^{2} H_{x}^{1}$ because in the estimate of 
\begin{equation*}
-\int u_{2} \partial_{2} u_{1} u_{1}
\end{equation*}  
within \eqref{estimate 5}, the diffusion for the Burgers' equation only gives $\lVert \nabla u_{1} \rVert_{L^{2}}^{2}$ so that we have no choice but to bound this by 
\begin{equation*}
\lVert u_{2} \rVert_{L^{\infty}} \lVert \partial_{2} u_{1} \rVert_{L^{2}} \lVert u_{1} \rVert_{L^{2}}. 
\end{equation*} 
However, while $\lVert \nabla u_{2} \rVert_{L^{2}} \in L_{T}^{2}$, this estimate will not go through immediately since $H^{1}(\mathbb{T}^{2}) \not\hookrightarrow L^{\infty}(\mathbb{T}^{2})$. Thus, Proposition \ref{Proposition 2.4} really is a special feature of the KSE due to the fourth-order diffusion. 
\end{remark} 

By Propositions \ref{Proposition 2.3}-\ref{Proposition 2.4} we have $u \in L_{T}^{\infty} L_{x}^{2}$ and thus we deduce that the 2D KSE is globally well-posed due to Theorem \ref{Theorem 2.1} as follows:   
\begin{theorem}\label{Theorem 2.7}
Suppose $u_{1} \in L_{T}^{r} W_{x}^{m,p}$ where $m \in [0,1)$ and 
\begin{equation*}
\begin{cases}
p \in [1, \frac{2}{m}), r \in (\frac{4}{3}, \frac{4}{1+m}] & \text{ if } m \in (0,1), \\
p \in (1, \infty], r \in [\frac{4}{3}, 4) & \text{ if } m = 0, 
\end{cases} 
\end{equation*} 
satisfy 
\begin{align*}
\frac{2}{p} + \frac{2}{r} = \frac{1}{p} + \frac{3+m}{2}. 
\end{align*} 
Then the 2D KSE is globally well-posed in $H^{1}(\mathbb{T}^{2})$. 
\end{theorem} 

\begin{remark}
As mentioned, Theorem \ref{Theorem 2.7} is very interesting because all previous component reduction results  relied on divergence-free condition (NSE, MHD system, and SQG equations). 

Another remarkable feature of Theorem \ref{Theorem 2.7} is that we obtained component reduction at no cost in terms of $\frac{3}{2} + \frac{1}{p}$ (cf. ``1'' in \eqref{Serrin criteria} and ``$\frac{5}{8}$'' in \eqref{estimate 12}). As mentioned,  reducing the Serrin criteria \eqref{Serrin criteria} to $u_{3}$ while retaining $\frac{3}{p} + \frac{2}{r} = 1$ is a well-known open problem that has caught much attention in the literature. 
\end{remark}

We are able to extend our result to the 3D case as follows:
\begin{proposition}\label{new proposition 1}
Suppose $u_{1}, u_{2} \in L_{T}^{r} W_{x}^{m,p}$ where $m \in [0, 1)$ and 
\begin{equation*}
\begin{cases}
p \in [\frac{3}{m+2}, \frac{3}{m}), r \in (\frac{4}{3}, 4] & \text{ if } m \in (0,1), \\
p \in [\frac{3}{2},\infty], r \in [\frac{4}{3}, 4] & \text{ if } m = 0, 
\end{cases} 
\end{equation*}
satisfy 
\begin{align*}
\frac{3}{p} + \frac{2}{r} = \frac{3}{2p} + \frac{3+m}{2}. 
\end{align*}
Then $u_{3} \in L_{T}^{\infty} L_{x}^{2} \cap L_{T}^{2} H_{x}^{2}$. 
\end{proposition} 

\begin{proof}
We consider the third component of \eqref{KSE}, take $L^{2}(\mathbb{T}^{3})$ inner-products with $u_{3}$ to obtain 
\begin{equation}\label{estimate 24} 
\frac{1}{2} \frac{d}{dt} \lVert u_{3} \rVert_{L^{2}}^{2} + \lVert \Delta u_{3} \rVert_{L^{2}}^{2} = - \int (u\cdot\nabla) u_{3} u_{3} + \lambda \lVert \nabla u_{3} \rVert_{L^{2}}^{2} 
\end{equation} 
where $\lambda \lVert \nabla u_{3} \rVert_{L^{2}}^{2} \leq \frac{1}{4} \lVert \Delta u_{3} \rVert_{L^{2}}^{2} + C \lVert u_{3} \rVert_{L^{2}}^{2}$ similarly to \eqref{estimate 2}, while 
\begin{align}
& - \int (u\cdot\nabla) u_{3} u_{3} \nonumber \\
=& - \int u_{1} \partial_{1} u_{3} u_{3} + u_{2} \partial_{2} u_{3} u_{3} \nonumber \\
\lesssim& ( \lVert u_{1} \rVert_{L^{\frac{3p}{3-mp}}} +\lVert u_{2} \rVert_{L^{\frac{3p}{3-mp}}}) \lVert u_{3} \rVert_{L^{\frac{6p}{3p + mp -3}}} \lVert \nabla u_{3} \rVert_{L^{\frac{6p}{3p + mp - 3}}} \nonumber \\
\lesssim&  (\lVert u_{1} \rVert_{W^{m,p}} + \lVert u_{2} \rVert_{W^{m,p}}) \lVert u_{3} \rVert_{L^{2}}^{\frac{ (3+m) p -3}{2p}} \lVert u_{3} \rVert_{H^{2}}^{\frac{ (1-m) p + 3}{2p}}  \nonumber \\
\leq& \frac{1}{4} \lVert\Delta u_{3} \rVert_{L^{2}}^{2} + C(\sum_{k=1}^{2} \lVert u_{k} \rVert_{W^{m,p}}^{\frac{4p}{(3+m) p -3}} + \lVert u_{k} \rVert_{W^{m,p}} ) \lVert u_{3} \rVert_{L^{2}}^{2} \label{estimate 25} 
\end{align}
by H$\ddot{\mathrm{o}}$lder's inequality, Sobolev embedding $L^{\frac{3p}{3-mp}} (\mathbb{T}^{3}) \hookrightarrow W^{m,p} (\mathbb{T}^{3})$ if $m > 0$, \eqref{GN1}-\eqref{GN2}, and Young's inequality. Applying these estimates to \eqref{estimate 24}, Gr\"onwall's inequality completes the proof. 
\end{proof}

\begin{proposition}\label{new proposition 2}
Suppose $u_{1}, u_{2} \in L_{T}^{r} W_{x}^{m,p}$ where $m \in [0, 1)$ and 
\begin{equation*}
\begin{cases}
p \in [\frac{3}{m+2}, \frac{3}{m}), r \in (\frac{4}{3}, 4] & \text{ if } m \in (0,1), \\
p \in [\frac{3}{2},\infty], r \in [\frac{4}{3}, 4] & \text{ if } m = 0, 
\end{cases} 
\end{equation*}
satisfy 
\begin{align*}
\frac{3}{p} + \frac{2}{r} = \frac{3}{2p} + \frac{3+m}{2}. 
\end{align*}
Then $u_{1}, u_{2} \in L_{T}^{\infty} L_{x}^{2} \cap L_{T}^{2} H_{x}^{2}$. \end{proposition} 

\begin{proof}
We fix $j \in \{1,2\}$ arbitrarily and take $L^{2}(\mathbb{T}^{3})$ inner product on the $j$-th component of \eqref{KSE} with $u_{j}$ to obtain 
\begin{equation}\label{estimate 26} 
\frac{1}{2} \frac{d}{dt} \lVert u_{j} \rVert_{L^{2}}^{2} + \lVert \Delta u_{j} \rVert_{L^{2}}^{2} = - \int (u\cdot\nabla) u_{j} u_{j} + \lambda \lVert \nabla u_{j} \rVert_{L^{2}}^{2}
\end{equation} 
where $\lambda \lVert \nabla u_{3} \rVert_{L^{2}}^{2} \leq \frac{1}{4} \lVert \Delta u_{3} \rVert_{L^{2}}^{2} + C \lVert u_{3} \rVert_{L^{2}}^{2}$ similarly to \eqref{estimate 2}, while 
\begin{align*}
- \int (u\cdot\nabla) u_{j} u_{j} = - \sum_{k=1}^{3} \int u_{k} \partial_{k} u_{j} u_{j}
\end{align*}
where $\int u_{j} \partial_{j} u_{j} u_{j} = 0$ so that for $k \in \{1,2\} \setminus \{j\}$, 
\begin{align}
 - \int (u\cdot\nabla) u_{j} u_{j} =& - \int u_{k} \partial_{k} u_{j} u_{j} + u_{3} \partial_{3} u_{j} u_{j} \nonumber \\
\lesssim& ( \lVert u_{k} \rVert_{L^{3}} + \lVert u_{3} \rVert_{L^{3}}) \lVert \nabla u_{j} \rVert_{L^{6}} \lVert u_{j} \rVert_{L^{2}} \nonumber \\
\leq& \frac{1}{4} \lVert \Delta u_{j} \rVert_{L^{2}}^{2} + C( 1+ \lVert u_{k} \rVert_{L^{3}}^{2} + \lVert u_{3} \rVert_{L^{2}}^{\frac{3}{2}} \lVert u_{3} \rVert_{H^{2}}^{\frac{1}{2}}) \lVert u_{j} \rVert_{L^{2}}^{2} \label{estimate 27} 
\end{align}
by H$\ddot{\mathrm{o}}$lder's inequality, Sobolev embedding $H^{1}(\mathbb{T}^{3}) \hookrightarrow L^{6} (\mathbb{T}^{3})$, Gagliardo-Nirenberg inequality, and Young's inequality. We apply these estimates to \eqref{estimate 26} and because 
\begin{align*}
& \int_{0}^{T} \lVert u_{k} \rVert_{L^{3}}^{2} \lesssim \int_{0}^{T} \lVert u_{k} \rVert_{W^{m,\frac{3}{m+1}}}^{2} d\tau < \infty, \\
&\int_{0}^{T} \lVert u_{3} \rVert_{L^{2}}^{\frac{3}{2}} \lVert u_{3} \rVert_{H^{2}}^{\frac{1}{2}} d\tau \leq \lVert u_{3} \rVert_{L_{T}^{\infty} L_{x}^{2}}^{\frac{3}{2}} \int_{0}^{T} \lVert u_{3} \rVert_{H^{2}}^{\frac{1}{2}}d\tau  < \infty
\end{align*}
due to Sobolev embedding $W^{m, \frac{3}{m+1}}(\mathbb{T}^{3}) \hookrightarrow L^{3}(\mathbb{T}^{3})$ if $m \in (0,1)$, hypothesis, and Proposition \ref{new proposition 1}, Gr\"onwall's inequality completes the proof. 
\end{proof}

Due to Propositions \ref{new proposition 1}-\ref{new proposition 2}, we see that $u \in L_{T}^{\infty} L_{x}^{2}$ and thus Theorem \ref{Theorem 2.1} again leads to the following theorem. 

\begin{theorem}\label{new_theorem}
Suppose $u_{1}, u_{2} \in L_{T}^{r} W_{x}^{m,p}$ where $m \in [0, 1)$ and 
\begin{equation*}
\begin{cases}
p \in [\frac{3}{m+2}, \frac{3}{m}), r \in (\frac{4}{3}, 4] & \text{ if } m \in (0,1), \\
p \in [\frac{3}{2},\infty], r \in [\frac{4}{3}, 4] & \text{ if } m = 0, 
\end{cases} 
\end{equation*}
satisfy 
\begin{align*}
\frac{3}{p} + \frac{2}{r} = \frac{3}{2p} + \frac{3+m}{2}. 
\end{align*}
Then the 3D KSE is globally well-posed in $H^{1}(\mathbb{T}^{3})$. \end{theorem} 

\begin{remark}
We leave it as an interesting problem to attain a regularity criterion for the 3D KSE in terms of one component instead of two components as we did in Theorem \ref{new_theorem}.
\end{remark}

\subsection{\texorpdfstring{Result on $\nabla u$}{}}\label{Subsection 3.3}
We aim to obtain a criterion in terms of $\lVert \nabla u \rVert_{L^{p}}$, motivated the case of the NSE in \eqref{Beirao criteria}. We can start from \eqref{1.6} and rely on \eqref{1.11}. For the nonlinear term we can write 
\begin{align}
\int (u\cdot\nabla) u \cdot \Delta u 
= - \sum_{i,j,k}\int\pnt{  \partial_{k} u_{i} \partial_{i} u_{j} \partial_{k} u_{j} - \frac{1}{2} \partial_{i} u_{i} (\partial_{k} u_{j})^{2} }
\lesssim \int \lvert \nabla u \rvert^{3}. \label{estimate 9}
\end{align} 
Now let us point out again that if we follow the standard approach of the NSE, then we would compute for $p \in [2, \infty]$, 
\begin{align*}
C\int \lvert \nabla u \rvert^{3} \lesssim& \lVert \nabla u \rVert_{L^{p}} \lVert \nabla u \rVert_{L^{2}} \lVert \nabla u \rVert_{L^{\frac{2p}{p-2}}} \nonumber\\
\lesssim&  \lVert \nabla u \rVert_{L^{p}} \lVert \nabla u \rVert_{L^{2}} \lVert \nabla u \rVert_{L^{2}}^{\frac{2p-N}{2p}} \lVert \nabla u \rVert_{H^{2}}^{\frac{N}{2p}} \nonumber \\
\leq& \frac{1}{4} \lVert \Delta \nabla u \rVert_{L^{2}}^{2} + C (\lVert \nabla u \rVert_{L^{p}}^{\frac{4p}{4p-N}} + \lVert \nabla u \rVert_{L^{p}}) \lVert \nabla u\rVert_{L^{2}}^{2} 
\end{align*} 
by H$\ddot{\mathrm{o}}$lder's, Gagliardo-Nirenberg, and Young's inequalities. This only leads to a criterion of $\nabla u \in L_{T}^{r}L_{x}^{p}$ where $\frac{N}{p} + \frac{2}{r} = 2 + \frac{N}{2p}$, $p \in [2,\infty], r \in [1, \frac{8}{8-N}]$. In fact, we approach this estimate differently from the NSE as follows in both cases $N \in\{ 2, 3\}$: 
\begin{align}
C \int \lvert \nabla u \rvert^{3} \lesssim&  \lVert \nabla u \rVert_{L^{p}} \lVert \nabla u \rVert_{L^{\frac{2p}{p-1}}}^{2}  \nonumber \\
\lesssim&  \lVert \nabla u \rVert_{L^{p}} \lVert \nabla u \rVert_{L^{2}}^{\frac{4p-N}{2p}} \lVert \nabla u \rVert_{H^{2}}^{\frac{N}{2p}}  \nonumber \\
\leq& \tfrac{1}{4} \lVert \Delta \nabla u \rVert_{L^{2}}^{2} + C (\lVert \nabla u \rVert_{L^{p}}^{\frac{4p}{4p-N}} + \lVert \nabla u \rVert_{L^{p}}) \lVert \nabla u \rVert_{L^{2}}^{2} \label{2.3a}
\end{align} 
by H$\ddot{\mathrm{o}}$lder's inequality, \eqref{GN1} and Young's inequality. We also continue from \eqref{star}
and estimate the nonlinear term as
\begin{align}
-\int (u\cdot\nabla) u \cdot u \leq& \lVert \nabla u \rVert_{L^{p}} \lVert u \rVert_{L^{\frac{2p}{p-1}}}^{2} \lesssim \lVert \nabla u \rVert_{L^{p}} \lVert u \rVert_{L^{2}}^{\frac{4p-N}{2p}} \lVert u\rVert_{H^{2}}^{\frac{N}{2p}} \nonumber \\
\leq& \frac{1}{8} \lVert \Delta u \rVert_{L^{2}}^{2} + C( \lVert \nabla u \rVert_{L^{p}}^{\frac{4p}{4p-N}} + \lVert \nabla u \rVert_{L^{p}}) \lVert u \rVert_{L^{2}}^{2} \label{estimate 8}
\end{align} 
by H$\ddot{\mathrm{o}}$lder's inequality \eqref{GN1} and Young's inequality. Considering \eqref{estimate 9}, \eqref{2.3a} and \eqref{estimate 8} in \eqref{1.6} and \eqref{star} gives us the following result. 

\begin{theorem}\label{Theorem 2.9}
Suppose $\nabla u \in L_{T}^{r}L_{x}^{p}$ where $p \in [1,\infty], r \in [1, \frac{4}{4-N}]$ satisfy 
\begin{equation*}
\frac{N}{p} + \frac{2}{r} = 2 + \frac{N}{2p}. 
\end{equation*}
Then the $N$D KSE is globally well-posed in $H^{1}(\mathbb{T}^{N})$. 
\end{theorem}
Taking $p = 1$ turns $2+ \frac{N}{2p}$ to $2 + \frac{N}{2}$. We emphasize again that $2 + \frac{N}{2}$ is remarkably larger than 2 in \eqref{Beirao criteria}.  
\subsection{\texorpdfstring{Results on $\nabla\cdot u$, and $\partial_{2}u_{2}$}{}}\label{Subsection 3.4}
Next, we aim to obtain a result on divergence $\nabla\cdot u$, which is an improvement of $\nabla u$, and in fact, even $\partial_{2}u_{2}$ at the optimal level. 

If we decide to work on the $H^{1}(\mathbb{T}^{N})$-estimate, then we can look at \eqref{1.6} again 
\begin{align*}
\frac{1}{2} \frac{d}{dt} \lVert \nabla u \rVert_{L^{2}}^{2} + \lVert \Delta \nabla u \rVert_{L^{2}}^{2} = \int (u\cdot\nabla) u \cdot \Delta u + \lambda \lVert \Delta u \rVert_{L^{2}}^{2}. 
\end{align*} 
Now if we integrate by parts, we can get 
\begin{equation*}
\int (u\cdot\nabla) u \cdot \Delta u = - \sum_{i,j,k = 1}^{N} \int \pnt{(\partial_{i}u_{i}) u_{j} \partial_{k}^{2} u_{j} + u_{i}u_{j}\partial_{i}\partial_{k}^{2}u_{j}}
\end{equation*} 
or 
\begin{equation*}
\int (u\cdot\nabla) u \cdot \Delta u = - \sum_{i,j,k = 1}^{N}\int \pnt{\partial_{k}u_{i}\partial_{i}u_{j}\partial_{k}u_{j} - \frac{1}{2} \partial_{i}u_{i} \lvert \partial_{k}u_{j} \rvert^{2}}.
\end{equation*} 
Although we have separated the divergence $\nabla\cdot u$ in $(\partial_{i}u_{i}) u_{j} \partial_{k}^{2} u_{j}$ or $\frac{1}{2} \partial_{i} u_{i} \lvert \partial_{k} u_{j} \rvert^{2}$, this will require regularity criteria of an additional term besides the divergence due to $u_{i}u_{j} \partial_{i} \partial_{k}^{2} u_{j}$ or $\partial_{k} u_{i} \partial_{i} u_{j} \partial_{k} u_{j}$. 

In fact, thanks to Theorem \ref{Theorem 2.1} we can work in $L^{2}(\mathbb{T}^{N})$-norm and avoid this issue. Indeed, we can restart from \eqref{star} where we can estimate $\lambda \lVert \nabla u \rVert_{L^{2}}^{2} \leq \frac{1}{4} \lVert \Delta u \rVert_{L^{2}}^{2} + C \lVert u \rVert_{L^{2}}^{2}$ identically to \eqref{estimate 2}, while  
\begin{equation*}
-\int (u\cdot\nabla) u \cdot u = - \int (u\cdot\nabla) \frac{1}{2} \lvert u \rvert^{2} = \frac{1}{2} \int (\nabla\cdot u ) \lvert u \rvert^{2}
\end{equation*} 
(note that this term vanishes for the NSE but not for the $N$D KSE when $N>1$). Now if we estimate 
\begin{equation*}
\int (\nabla\cdot u) \lvert u \rvert^{2} \leq \lVert \nabla\cdot u \rVert_{L^{\infty}} \lVert u \rVert_{L^{2}}^{2}, 
\end{equation*} 
we can get a criterion in terms of $\int_{0}^{T} \lVert \nabla \cdot u \rVert_{L^{\infty}} d \tau$ analogously to the 3D Euler equations (\cite{Beale_Kato_Majda_1984}). But to make full use of the fourth-order diffusion, we compute similarly to \eqref{2.3a}
\begin{align}
 \frac{1}{2} \int (\nabla\cdot u) \lvert u \rvert^{2} \leq& \frac{1}{2} \lVert \nabla \cdot u \rVert_{L^{p}} \lVert u \rVert_{L^{\frac{2p}{p-1}}}^{2} \lesssim \lVert \nabla \cdot u \rVert_{L^{p}} \lVert u \rVert_{L^{2}}^{\frac{4p-N}{2p}} \lVert u \rVert_{H^{2}}^{\frac{N}{2p}} \nonumber \\
\leq& \frac{1}{4} \lVert \Delta u \rVert_{L^{2}}^{2} + C (\lVert \nabla \cdot u \rVert_{L^{p}}^{\frac{4p}{4p-N}} + \lVert \nabla\cdot u \rVert_{L^{p}}) \lVert u \rVert_{L^{2}}^{2} \label{2.3b}
\end{align} 
by H$\ddot{\mathrm{o}}$lder's inequality, \eqref{GN1}, and Young's inequality. Because Theorem \ref{Theorem 2.1} implies that $L_{T}^{\infty} L_{x}^{2}$-bound leads immediately to the global well-posedness of the $N$D KSE, we obtain the following result. 

\begin{theorem}\label{Theorem 3.10}
Suppose $\nabla\cdot u \in L_{T}^{r}L_{x}^{p}$ where $p \in [1,\infty], r \in [1, \frac{4}{4-N}]$ satisfy 
\begin{equation*}
\frac{N}{p} + \frac{2}{r} = 2 + \frac{N}{2p}.    
\end{equation*}
Then the $N$D KSE is globally well-posed in $H^{1}(\mathbb{T}^{N})$. 
\end{theorem}

By a similar trick to the proof of Theorem \ref{Theorem 2.7}, we can also obtain the following result in the 2D case:
\begin{proposition}\label{Proposition 2.11}
Suppose $\partial_{2} u_{2} \in L_{T}^{r}L_{x}^{p}$ where $p \in [1,\infty], r \in [1,2]$ satisfy 
\begin{equation*}
\frac{2}{p} + \frac{2}{r} = 2 + \frac{1}{p}.     
\end{equation*}
Then $u_{1} \in L_{T}^{\infty} L_{x}^{2} \cap L_{T}^{2} H_{x}^{2}$.
\end{proposition} 

\begin{proof}
This is almost identical to the previous computation for the criteria of $\nabla\cdot u$. We restart from \eqref{2.2b} where $\lambda \lVert \nabla u_{1} \rVert_{L^{2}}^{2}$ may be handled again identically to \eqref{estimate 2} while similarly to \eqref{2.3b}
\begin{align}
- \int (u\cdot\nabla) u_{1} u_{1} =&  \frac{1}{2} \int \partial_{2}u_{2}  (u_{1})^{2} \leq \frac{1}{2} \lVert \partial_{2}u_{2} \rVert_{L^{p}} \lVert u_{1} \rVert_{L^{\frac{2p}{p-1}}}^{2} \nonumber \\
\lesssim&  \lVert \partial_{2}u_{2} \rVert_{L^{p}} \lVert u_{1} \rVert_{L^{2}}^{\frac{2p-1}{p}} \lVert u_{1} \rVert_{H^{2}}^{\frac{1}{p}} \nonumber \\
\leq& \frac{1}{4} \lVert \Delta u_{1} \rVert_{L^{2}}^{2} + C (\lVert \partial_{2} u_{2} \rVert_{L^{p}}^{\frac{2p}{2p-1}} + \lVert \partial_{2} u_{2} \rVert_{L^{p}}) \lVert u_{1} \rVert_{L^{2}}^{2}. \label{2.4} 
\end{align} 
Now applying \eqref{2.4} to \eqref{2.2b} completes the proof of Proposition \ref{Proposition 2.11}. 
\end{proof} 

We can raise this regularity immediately as follows: 
\begin{proposition}\label{Proposition 2.12}
Suppose $\partial_{2} u_{2} \in L_{T}^{r}L_{x}^{p}$ where $p \in [1,\infty], r \in [1,2]$ satisfy 
\begin{equation*}
\frac{2}{p} + \frac{2}{r} = 2 + \frac{1}{p}.    
\end{equation*}
Then $u_{2} \in L_{T}^{\infty} L_{x}^{2} \cap L_{T}^{2} H_{x}^{2}$.
\end{proposition} 

\begin{proof}
We restart from \eqref{2.2o} again where $\lambda \lVert \nabla u_{2} \rVert_{L^{2}}^{2}$ may be handled identically to \eqref{estimate 2} while 
\begin{align}
- \int (u\cdot\nabla) u_{2} u_{2} =& - \int u_{1} \partial_{1} u_{2} u_{2} \leq \lVert u_{1} \rVert_{L^{\infty}} \lVert \partial_{1} u_{2} \rVert_{L^{2}} \lVert u_{2} \rVert_{L^{2}} \nonumber \\
\lesssim&  (\lVert \nabla u_{2} \rVert_{L^{2}}^{2} +  \lVert u_{1} \rVert_{L^{\infty}}^{2} \lVert u_{2} \rVert_{L^{2}}^{2}) \nonumber\\
\lesssim& ( \lVert u_{2} \rVert_{L^{2}} \lVert \Delta u_{2} \rVert_{L^{2}} +  \lVert u_{1} \rVert_{H^{2}}^{2} \lVert u_{2} \rVert_{L^{2}}^{2}) \nonumber  \\
\leq& \frac{1}{4} \lVert \Delta u_{2} \rVert_{L^{2}}^{2} +C (1+ \lVert u_{1} \rVert_{H^{2}}^{2}) \lVert u_{2} \rVert_{L^{2}}^{2} \label{2.4a} 
\end{align}
by the Gagliardo-Nirenberg inequality as in \eqref{estimate 2}, the Sobolev embedding  $H^{2}(\mathbb{T}^{2})\hookrightarrow L^{\infty}(\mathbb{T}^{2})$, and Young's inequality. Therefore, applying \eqref{2.4a} to \eqref{2.2o} leads to 
\begin{equation*}
\frac{1}{2} \frac{d}{dt} \lVert u_{2} \rVert_{L^{2}}^{2}+ \lVert \Delta u_{2} \rVert_{L^{2}}^{2} \leq \frac{1}{2} \lVert \Delta u_{2} \rVert_{L^{2}}^{2}+ C (1+ \lVert u_{1} \rVert_{H^{2}}^{2}) \lVert u_{2} \rVert_{L^{2}}^{2}. 
\end{equation*} 
Now applying Proposition \ref{Proposition 2.11} completes the proof of Proposition \ref{Proposition 2.12}. 
\end{proof}

By Propositions \ref{Proposition 2.11}-\ref{Proposition 2.12}, we deduce that $u \in L_{T}^{\infty} L_{x}^{2}$ and thus by Theorem \ref{Theorem 2.1} we see that the 2D KSE is globally well-posed in $H^{1}(\mathbb{T}^{2})$ as follows: 

\begin{theorem}\label{Theorem 2.13}
Suppose $\partial_{2} u_{2} \in L_{T}^{r}L_{x}^{p}$ where $p \in [1,\infty], r \in [1,2]$ satisfy 
\begin{equation*}
\frac{2}{p} + \frac{2}{r} = 2 + \frac{1}{p}.
\end{equation*}
Then 2D KSE \eqref{KSE} is globally well-posed in $H^{1}(\mathbb{T}^{2})$.
\end{theorem} 

\begin{remark}
We emphasize again that Theorem \ref{Theorem 2.13} is surprising due to the KSE lacking divergence-free property and that $2 + \frac{1}{p}$ in Theorem \ref{Theorem 2.13} is same as the condition in Theorem \ref{Theorem 2.9}; in comparison, ``$\frac{3(p+2)}{4p}$ for $p  > 2$'' in \eqref{estimate 13} is much smaller than ``2'' in \eqref{Beirao criteria}. 
\end{remark}

Lastly, we aim for $\partial_{2}u_{1}$. It is well-known that in general, such a criterion is expected to be more difficult than $\partial_{2}u_{2}$ where the components of $\nabla$ and $u$ match; indeed, $\frac{p+3}{2p} < \frac{3(p+2)}{4p}$ in \eqref{estimate 13}-\eqref{estimate 14}. Now if we return again to \eqref{star} where we can rely on the same estimates in \eqref{estimate 2} to handle the term $\lambda \lVert \nabla u \rVert_{L^{2}}^{2}$, then we are only faced with a nonlinear term of 
\begin{align*}
\int (u\cdot\nabla) u \cdot u = - \int \pnt{(u\cdot\nabla) u_{1} u_{1} + (u\cdot\nabla) u_{2} u_{2}}. 
\end{align*} 
The first term allows us to separate $\partial_{2}u_{1}$ as follows:  
\begin{align}
-\int (u\cdot\nabla) u_{1} u_{1} =   -\int\pnt{ u_{1}\partial_{1}u_{1}u_{1} + u_{2} \partial_{2} u_{1} u_{1}} = - \int u_{2}\partial_{2} u_{1} u_{1};  \label{estimate 4} 
\end{align}
however, for the second term we are faced with 
\begin{align*}
\int (u\cdot\nabla) u_{2} u_{2} = - \int \pnt{u_{1} \partial_{1} u_{2} u_{2} + u_{2} \partial_{2} u_{2} u_{2}} = - \int u_{1} \partial_{1} u_{2} u_{2}.
\end{align*}
We seem to face a significant difficulty here. We can integrate by parts to deduce $\frac{1}{2} \int \partial_{1} u_{1} u_{2} u_{2}$, but $\partial_{2} u_{1}$ is nowhere to be found. One idea would be to rely on some analogue of anisotropic Sobolev inequality that has proven to be useful in the theory of fluid PDEs such as the NSE and the MHD system (e.g., \cite[Lemma 1.1]{CW11}). However, those results are typically all on the whole space such as $\mathbb{R}^{2}$ because they rely on the decay at infinity; i.e., while 
\begin{equation*}
f(x) = \int_{-\infty}^{x} \partial_{z} f(z) dz, 
\end{equation*}  
an analogous attempt in $\mathbb{T}^{2}$ gives 
\begin{equation*}
f(x) = \int_{-\pi}^{x} \partial_{z} f(z) dz + f(-\pi). 
\end{equation*} 
Remarkably, there is a trick that seems to be special to the KSE here. Let us make our argument precise by working on the equation of $\phi$ rather than $u$ and attain a criterion in terms of $\partial_{12} \phi$, informally $\partial_{1} u_{2}$ as desired (see Theorem \ref{Theorem 3.14}). 

\subsection{\texorpdfstring{Results on $\phi$ and $\partial_{12} \phi$}{}}\label{Subsection 3.5} 
We emphasize that the following two results focus on \eqref{KSE_scalar}, independently of \eqref{KSE}. Let us start with a criterion in terms of $\phi$. 
\begin{theorem}\label{Theorem 3.17}
Suppose $\phi \in L_{T}^{r}W_{x}^{m,p}$ where $m \in [0, 1)$, and 
\begin{align*}
\begin{cases}
p \in [1, \frac{N}{m}), r \in (2, \frac{4}{m}] & \text{ if } N = 2, m \in (0,1), \\
p \in (\frac{N}{m+2}, \frac{N}{m}), r \in (2,\infty) & \text{ if } N = 3, m \in (0,1), \\
p \in (\frac{N}{2}, \infty], r \in [2, \infty) & \text{ if } N \in \{2,3\}, m = 0.
\end{cases} 
\end{align*}
satisfy 
\begin{align*}
\frac{N}{p} + \frac{2}{r} = \frac{2+m}{2} + \frac{N}{2p}. 
\end{align*}
Then the $N$D KSE \eqref{KSE_scalar} is globally well-posed in $L^{2}(\mathbb{T}^{N})$. 
\end{theorem} 

\begin{remark}
Very recently, Feng and Mazzucato \cite{Feng_Mazzucato_2020} showed a blowup criterion of $\phi \in L_{T}^{\infty} L_{x}^{2}$ for 2D KSE. If $N = 2$, we only require $\frac{2}{p} + \frac{2}{r} \leq 1 + \frac{1}{p}$, allowing cases such as $L_{T}^{\infty} L_{x}^{p}$ for any $p > 1$ and $L_{T}^{4}L_{x}^{2}$. 
\end{remark} 

\begin{proof}
We take $L^{2}(\mathbb{T}^{N})$-inner products on \eqref{KSE_scalar} with $\phi$ to deduce 
\begin{equation}\label{L2phi}
\frac{1}{2} \frac{d}{dt} \lVert \phi \rVert_{L^{2}}^{2} + \lVert \Delta \phi \rVert_{L^{2}}^{2} = - \lambda \int \Delta \phi \phi - \frac{1}{2} \int \lvert \nabla \phi \rvert^{2} \phi. 
\end{equation} 
Now 
\begin{equation}\label{estimate 15}
-\int \Delta \phi \phi \leq \lVert \Delta \phi \rVert_{L^{2}}\lVert \phi \rVert_{L^{2}} \leq \frac{1}{4} \lVert \Delta \phi \rVert_{L^{2}}^{2} + C \lVert \phi \rVert_{L^{2}}^{2}
\end{equation} 
by Young's inequality, while 
\begin{align}
- \frac{1}{2} \int \lvert \nabla \phi \rvert^{2} \phi \leq& \frac{1}{2} \lVert \phi \rVert_{L^{\frac{Np}{N- mp}}} \lVert \nabla \phi \rVert_{L^{\frac{2N p}{Np + mp - N}}}^{2}  \nonumber\\
\lesssim& \lVert \phi \rVert_{W^{m,p}} \lVert \phi \rVert_{L^{2}}^{\frac{ (m+2) p - N}{2p}} \lVert \phi \rVert_{H^{2}}^{\frac{ (2-m) p + N}{2p}} \nonumber  \\
\leq& \frac{1}{4} \lVert \Delta \phi \rVert_{L^{2}}^{2} + C\Big( \lVert \phi \rVert_{W^{m,p}}^{\frac{4p}{(2+m) p- N}} + \lVert \phi \rVert_{W^{m,p}}\Big) \lVert \phi \rVert_{L^{2}}^{2} \label{estimate 16} 
\end{align}
by H$\ddot{\mathrm{o}}$lder's inequality, Sobolev embedding $W^{m,p} (\mathbb{T}^{N}) \hookrightarrow L^{\frac{Np}{N-mp}}(\mathbb{T}^{N})$ in case $m \in (0,1)$ so that $p <\frac{N}{m}$ by hypothesis, \eqref{GN2}, and Young's inequality. By applying \eqref{estimate 15}-\eqref{estimate 16} to \eqref{L2phi}, we deduce that $\phi \in L_{T}^{\infty} L_{x}^{2} \cap L_{T}^{2} H_{x}^{2}$ and thus the desired result. 
\end{proof}

In the following theorem, we assume $N = 2$ and aim for a result analogous to a criterion of \eqref{KSE} in terms of $\partial_{1}u_{2}$ as promised. 
\begin{theorem}\label{Theorem 3.14}
Suppose $\partial_{12} \phi \in L_{T}^{r}L_{x}^{p}$ where $p \in [1,\infty], r \in [1,2]$ satisfy 
\begin{equation*}
\frac{2}{p} + \frac{2}{r} = 2 + \frac{1}{p}.     
\end{equation*}
Then 2D KSE \eqref{KSE_scalar} is globally well-posed in $L^{2}(\mathbb{T}^{2})$.
\end{theorem} 

\begin{proof}
Let us consider the equation \eqref{KSE_scalar}, take $L^{2}(\mathbb{T}^{2})$-inner products with $-\Delta \phi$ so that 
\begin{align}
& \frac{1}{2} \frac{d}{dt} \lVert \nabla \phi \rVert_{L^{2}}^{2} + \lVert \Delta \nabla \phi \rVert_{L^{2}}^{2} \nonumber \\
=&\nonumber
\lambda \lVert \Delta \phi \rVert_{L^{2}}^{2} - \frac{1}{2} \int \partial_{1} ((\partial_{1} \phi)^{2} + (\partial_{2} \phi)^{2}) \partial_{1} \phi 
-\frac{1}{2} \int\partial_{2} ((\partial_{1} \phi)^{2} + (\partial_{2} \phi)^{2}) \partial_{2} \phi 
\\& \triangleq
\lambda \lVert \Delta \phi \rVert_{L^{2}}^{2} +
 I + II. \label{estimate 18}
\end{align} 
It is straight-forward to estimate 
\begin{equation}\label{estimate 22}
\lambda \lVert \Delta \phi \rVert_{L^{2}}^{2} \leq \frac{1}{8} \lVert \Delta \nabla \phi \rVert_{L^{2}}^{2} + C \lVert \nabla \phi \rVert_{L^{2}}^{2} 
\end{equation}
by Young's inequality. Next, we estimate 
\begin{align}
I =& - \int \frac{1}{3} \partial_{1} (\partial_{1} \phi)^{3} + \partial_{2} \phi \partial_{12} \phi \partial_{1} \phi \label{estimate 19} \\
\leq& \lVert \partial_{2} \phi \rVert_{L^{\frac{2p}{p-1}}} \lVert \partial_{12} \phi \rVert_{L^{p}} \lVert \partial_{1} \phi \rVert_{L^{\frac{2p}{p-1}}} \nonumber\\
\lesssim& \lVert \partial_{2} \phi \rVert_{L^{2}}^{\frac{2p-1}{2p}} \lVert \partial_{2} \phi \rVert_{H^{2}}^{\frac{1}{2p}} \lVert \partial_{12} \phi \rVert_{L^{p}} \lVert \partial_{1} \phi \rVert_{L^{2}}^{\frac{2p-1}{2p}} \lVert \partial_{1} \phi \rVert_{H^{2}}^{\frac{1}{2p}}\nonumber \\
\leq& \frac{1}{8} (\lVert \partial_{1} \Delta \phi \rVert_{L^{2}}^{2} + \lVert \partial_{2} \Delta \phi \rVert_{L^{2}}^{2}) \nonumber\\
&+ C( \lVert \partial_{12} \phi \rVert_{L^{p}} + \lVert \partial_{12} \phi \rVert_{L^{p}}^{\frac{4p}{4p-1}} + \lVert \partial_{12} \phi \rVert_{L^{p}}^{\frac{2p}{2p-1}}) \lVert \nabla \phi \rVert_{L^{2}}^{2} \nonumber
\end{align} 
by H$\ddot{\mathrm{o}}$lder's inequality, \eqref{GN1} and Young's inequality. Next, we get a break here as it turns out  that we can apply identical estimates in $I$ to $II$ because 
\begin{align}
II =& - \frac{1}{2} \int \partial_{2} ((\partial_{1} \phi)^{2} + (\partial_{2} \phi)^{2} ) \partial_{2} \phi \nonumber\\
=& - \int [(\partial_{1} \phi) (\partial_{12} \phi) + (\partial_{2} \phi) (\partial_{22} \phi) ] \partial_{2} \phi = - \int \partial_{1} \phi \partial_{12} \phi \partial_{2} \phi. \label{estimate 20}
\end{align}
By applying \eqref{estimate 22}-\eqref{estimate 20} to \eqref{estimate 18}, we obtain $\phi \in L_{T}^{\infty} \dot{H}_{x}^{1} \cap L_{T}^{2} \dot{H}_{x}^{3}$. It is not difficult to ensure the lower regularity at $L_{T}^{\infty}L_{x}^{2}$ from such high regularity by going back to \eqref{L2phi}. Indeed, we can compute from \eqref{L2phi}
\begin{align}
 \frac{1}{2} \frac{d}{dt} \lVert \phi \rVert_{L^{2}}^{2} &+ \lVert \Delta \phi \rVert_{L^{2}}^{2}  \leq \lambda \lVert \nabla \phi \rVert_{L^{2}}^{2} + \frac{1}{2} \lVert \phi \rVert_{L^{2}} \lVert \nabla \phi \rVert_{L^{2}} \lVert \nabla \phi \rVert_{L^{\infty}} \nonumber\\
\lesssim& \lVert \phi \rVert_{L^{2}} \lVert \Delta \phi \rVert_{L^{2}} + \lVert\phi \rVert_{L^{2}} \lVert \nabla \phi \rVert_{L^{2}} (\lVert \nabla \phi \rVert_{L^{2}} + \lVert \Delta \nabla \phi \rVert_{L^{2}}) \label{estimate 21}
\end{align} 
where we used the same estimate in \eqref{estimate 2} and the Sobolev embedding  $H^{2}(\mathbb{T}^{2}) \hookrightarrow L^{\infty}(\mathbb{T}^{2})$. This estimate \eqref{estimate 21} immediately implies 
\begin{align}
\frac{d}{dt} \lVert \phi \rVert_{L^{2}} \lesssim  \lVert \nabla \phi \rVert_{L^{2}} + \lVert \Delta \nabla \phi \rVert_{L^{2}} + \lVert \nabla \phi \rVert_{L^{2}} (\lVert \nabla \phi \rVert_{L^{2}} + \lVert \Delta \nabla \phi \rVert_{L^{2}}) \label{estimate 23} 
\end{align} 
due to Young's inequality. Integrating \eqref{estimate 23} over $[0,t]$ now and relying on the fact that $\phi \in L_{T}^{\infty} \dot{H}_{x}^{1} \cap L_{T}^{2} \dot{H}_{x}^{3}$ complete the proof of Theorem \ref{Theorem 3.14}. 
\end{proof}

\section{Computational Results}\label{sec_comp}
In this section, we present some computational results.  In particular, we simulate solutions to the 2D KSE.  All simulations are done for the scalar form of the system \eqref{KSE_scalar},
(except for the simulation used to generate Figure \ref{fig_Pu}, which directly simulated \eqref{KSE}),
with the slight modification of subtracting the spatial average at each time step, as is common in simulations of the KSE (see, e.g., \cite{Kalogirou_Keaveny_Papageorgiou_2015}).  Note that this modification formally preserves the differential form of the KSE, system \eqref{KSE}.  In this section, we assume that $u=\nabla\phi$;
that is, here, $u$ is essentially a notation for $\nabla \phi$.

The existing literature on computational studies of the 2D KSE is fairly limited, but we refer to \cite{Kalogirou_Keaveny_Papageorgiou_2015} where an exhaustive computational study of the KSE was carried out in various rectangular domains.  Various modifications to the 2D KSE have been investigated computationally in several works; see, e.g., \cite{Tomlin_Kalogirou_Papageorgiou_2018,Larios_Yamazaki_2020_rKSE}.

Our simulations were computed via MATLAB (R2020a) on the periodic square domain $[-\pi,\pi)^2$, using a uniform rectangular grid $256\times256$.  Derivatives were computed using spectral methods (also called pseudo-spectral methods), specifically using Matlab's \texttt{fftn} and \texttt{ifftn} implementations of the $N$D discrete Fourier transform, making sure to respect the 2/3's dealiasing law for quadratic nonlinearities, as well as the vanishing of the Nyquist frequency for only odd-order derivatives.  Time stepping was carried out using a 4th-order ETD-RK4 method (see, e.g., \cite{Kassam_Trefethen_2005,Kennedy_Carpenter_2003_IMEX}) to handle the linear terms implicitly, while the nonlinear term was computed explicitly and in physical space. In particular, a complex contour with 64 nodes (technically 32 nodes plus a symmetry condition) was used to handle the removable singularities in the ETD-RK4 coefficients, as proposed in \cite{Kassam_Trefethen_2005}. We used a uniform time-step of $\Delta t = 2.1227\times 10^{-5}$, well below the advective CFL given by $\Delta t< \Delta x/(\tfrac12\|\nabla\phi\|_{L^\infty(0,T;L^\infty(\nT^2))})\approx4.4\times 10^{-5}$.  All simulations presented here were well-resolved at each time step, in the sense that the magnitude of the energy spectrum of $\vphi$ was below machine precision ($\epsilon=2.2204\times 10^{-16}$ in MATLAB) for all wave numbers above the $2/3$'s dealising cutoff (i.e., for all $\vec{k}\in\nZ^2$ with $|\vec{k}|\geq256/3$). This can be seen in the plots of the spectrum at each time step in Figure \ref{fig_spec}. We used roughly the largest $\lambda$ possible at this resolution ($256^2$) while maintaining well-resolvedness; namely $\lambda=29.1$ (found experimentally), meaning that 96 modes were unstable (not counting the zero mode, which was set to zero); that is, $\#\{\vec{k}\in\nZ^2\setminus\{\vec{0}\}:|\vec{k}|^2\triangleq k_1^2+k_2^2<29.1\}=96$.  However, we observed in tests that the behavior of the solution presented here was qualitatively similar for a wide range of $\lambda$ values.

For concreteness, we use the initial data $\phi^{in}$ studied in \cite{Kalogirou_Keaveny_Papageorgiou_2015}; namely 
\begin{align}\label{KKP}
 \phi^{in}(x,y) = \sin(x+y) + \sin(x) + \sin(y).
\end{align}
However, we also tested several other choices of initial data, including data based on randomly generated Fourier amplitudes, and found qualitatively similar results to those presented here.  Thus, for simplicity of presentation, we only include the results from simulations with initial data \eqref{KKP}.
\begin{figure}[!ht]
\centering
\captionsetup{width=\textwidth}
\includegraphics[width=0.7\textwidth,trim=1mm 3mm 22mm 8mm, clip]
{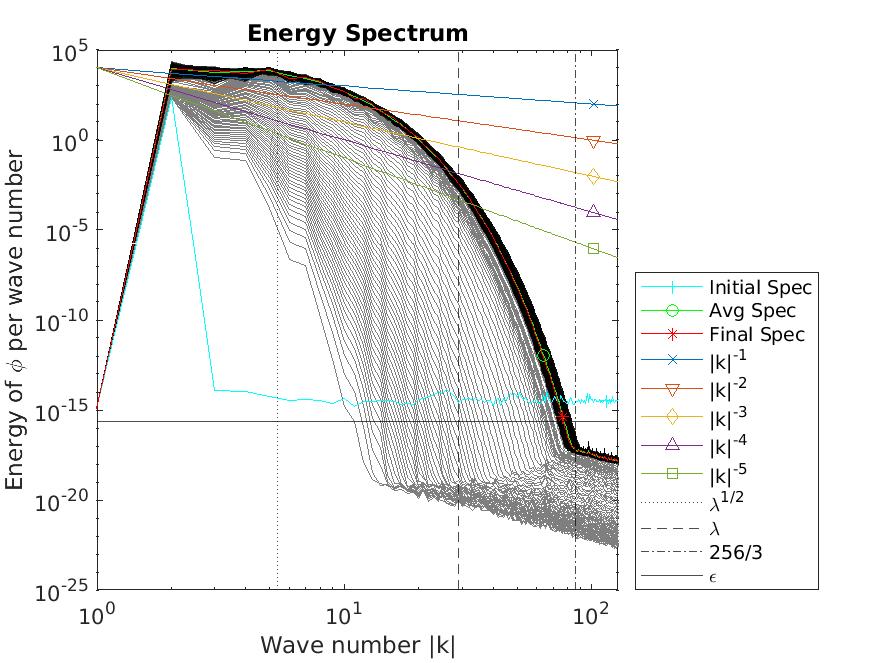}
\vspace{-4mm}
\caption{\label{fig_spec} \scriptsize The spectrum at time steps $t=0.0,t=0.001,t=0.002,\ldots,10.0.$ (log-log scale), starting in grey and getting darker with increasing time.  Highlighted are the initial spectrum  (cyan {\color{cyan} $+$}), the final ($t=10.0$) spectrum (red {\color{red}$*$}), and the average spectrum (green {\color{green} $\circ$}), (averaged over $9\leq t\leq10$).  }
\end{figure}

We computed the $L^p$ norms for $1\leq p\leq\infty$ in physical space (using Riemann sums for $p<\infty$), with the modification of dividing by the area of the domain; that is, in this section only, unless otherwise noted, we denote
\begin{align}\label{normalizeNorms}
 \|f\|_{L^p}^p \triangleq \frac{1}{4\pi^2}\int_{-\pi}^\pi\int_{-\pi}^\pi |f(x,y)|^p\,dx\,dy.
\end{align}
for $p<\infty$, so that $\|f\|_{L^p}\leq\|f\|_{L^\infty}\triangleq\esssup_{(x,y)\in\nT^2}|f(x,y)|$.  This is done only for aesthetic purposes so that it is easier to see the plots on the same figure, see Figures \ref{LpNorms1} and \ref{LpNorms2}.  From these graphics, it appears that all the criteria proven in the present work hold.  Indeed, it appears that $\phi$, $\phi_x$, $\phi_y$, $u_1\triangleq\phi_x$, $u_2\triangleq\phi_y$, $\phi_{xx}$, $\phi_{xy}\equiv\phi_{yx}$, $\phi_{yy}$, and $\text{div}(u)=\Delta\phi$ are all in $L^\infty([0,T],L^\infty)$ for all $T>0$, at least for the initial data and $\lambda$-values we tested.  Thus, in these tests, we do not see evidence for ill-posedness.

\begin{figure}[!ht]
\centering
\includegraphics[width=0.32\textwidth,trim=22mm 8mm 25mm 4mm, clip]
{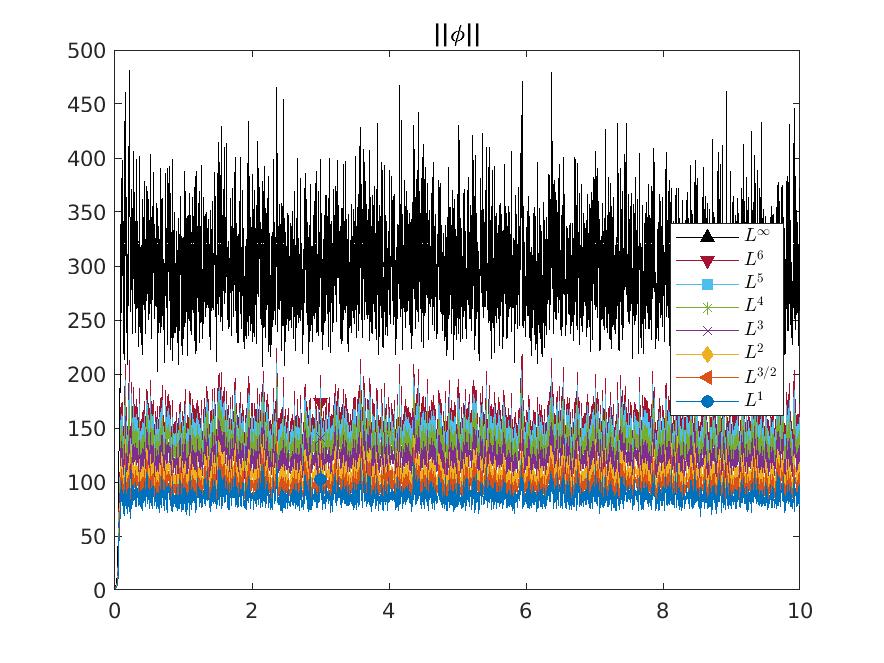}
\includegraphics[width=0.32\textwidth,trim=20mm 8mm 25mm 4mm, clip]
{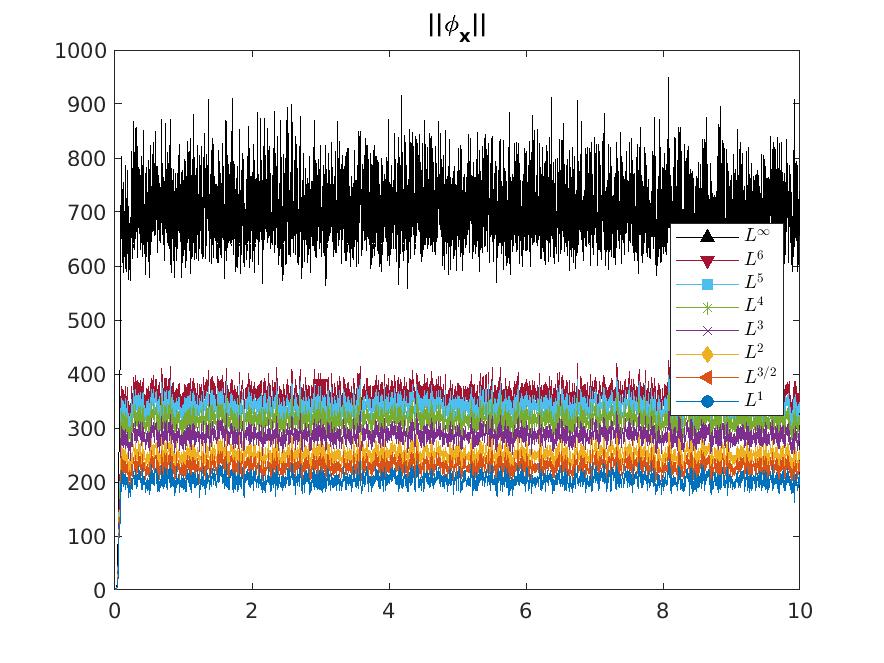}
\includegraphics[width=0.32\textwidth,trim=20mm 8mm 25mm 4mm, clip]
{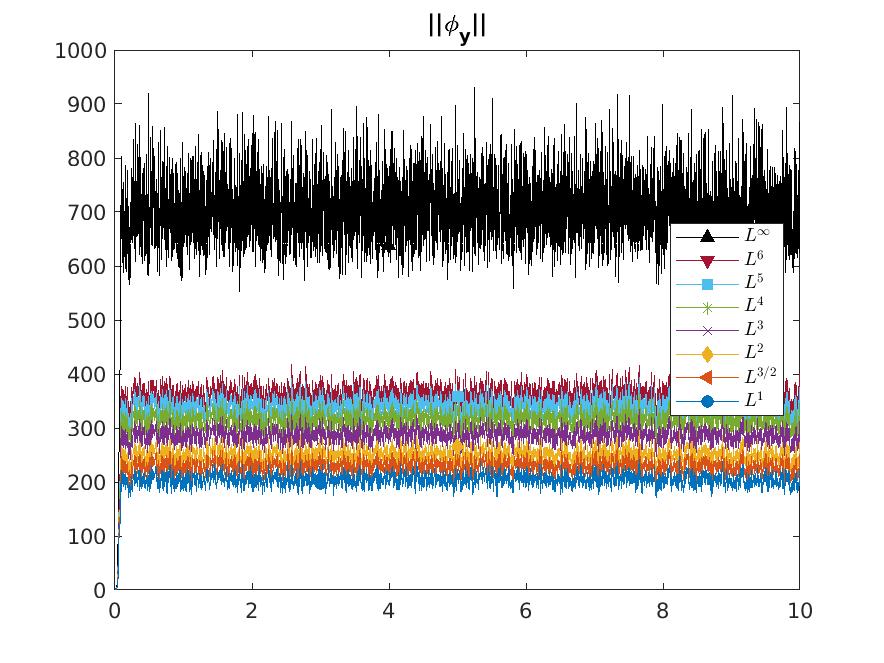}
\caption{\label{LpNorms1} \scriptsize $L^p$ norms of $\phi$ (left), $\phi_x$ (middle) and $\phi_y$ (right).  Note that $u_1=\phi_x$, and $u_2=\phi_y$.  ($p = 1,\tfrac32,2,3,4,5,6,\infty.$)}
\end{figure}

\begin{figure}[!ht]
\centering
\includegraphics[width=0.49\textwidth,trim=19mm 8mm 25mm 4mm, clip]
{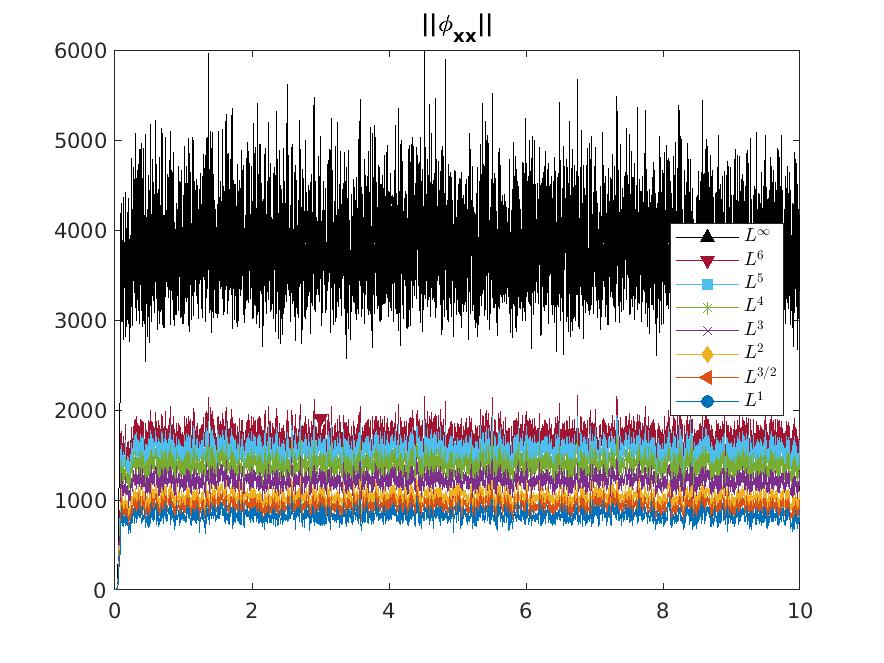}
\includegraphics[width=0.49\textwidth,trim=19mm 8mm 25mm 4mm, clip]
{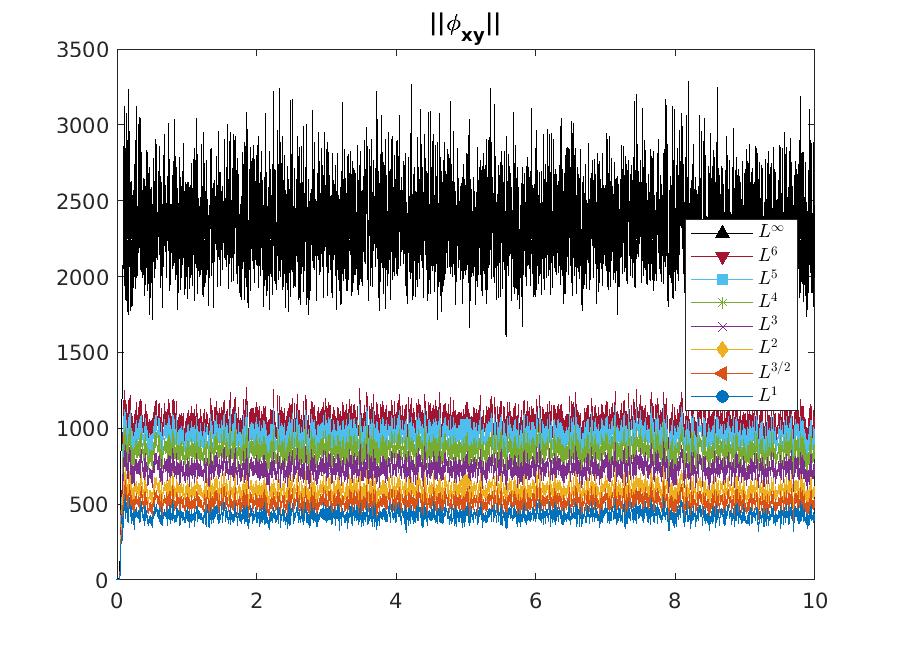}\\
\includegraphics[width=0.49\textwidth,trim=19mm 8mm 25mm 4mm, clip]
{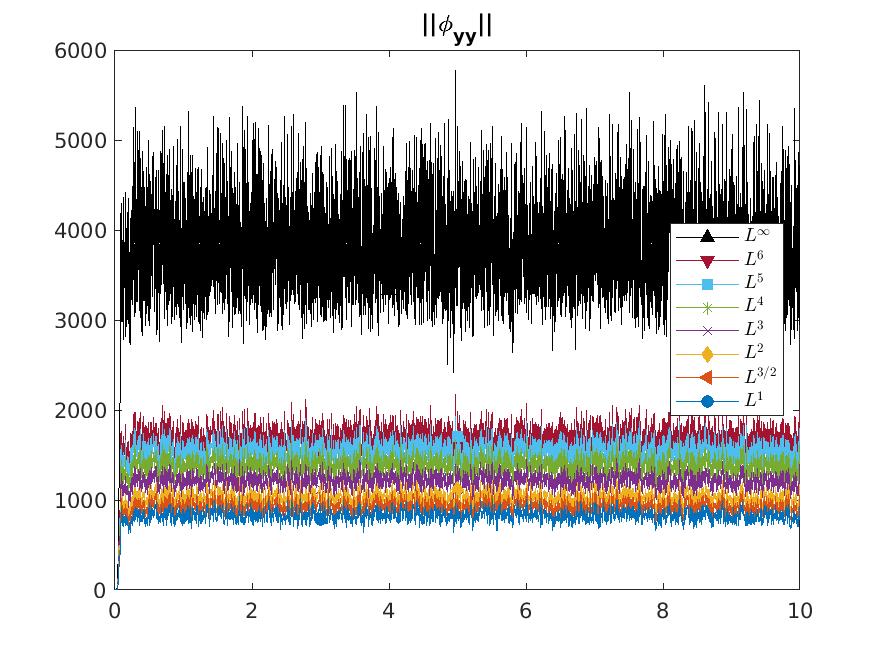}
\includegraphics[width=0.49\textwidth,trim=19mm 8mm 25mm 4mm, clip]
{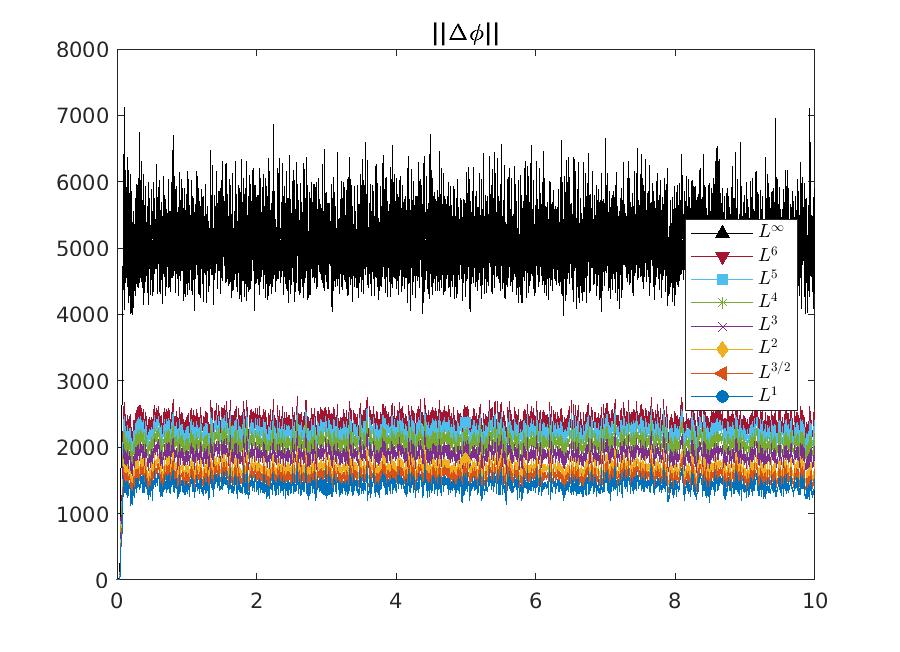}
\caption{\label{LpNorms2} \scriptsize $L^p$ norms of $\phi_{xx}$ (top left), $\phi_{xy}$ (top right) and $\phi_{yy}$ (bottom left), and $\Delta\phi$ (bottom right).  Note that $\text{div}(u)=\Delta\phi$. ($p = 1,\tfrac32,2,3,4,5,6,\infty.$)}
\end{figure}

In the context of the 2D NSE, it is common to look for regions in, e.g., the energy-enstrophy plane or enstrophy-palenstrophy plane, etc., where the attractor lies (see, e.g., \cite{Cao_Jolly_Titi_Whitehead_2019,Dascaliuc_Foias_Jolly_2005,Dascaliuc_Foias_Jolly_2007,Dascaliuc_Foias_Jolly_2008, Dascaliuc_Foias_Jolly_2010,Emami_Bowman_2018}).  In the same spirit, we consider analogous planes combining  $\|\vphi\|_L^2$ (``Energy''),  $\|\nabla\vphi\|_L^2$ (``Enstrophy''), and  $\|\Delta\vphi\|_L^2$ (``Palenstrophy''), where we have borrowed the names for these quantities from the NSE setting even though they have different physical meanings in the context of the KSE.  In Figure \ref{EnEnsPal}, these plots are displayed.  Here, we see that the large-time dynamics are constrained to a relatively small region of these planes with a roughly elliptical shape.  

\begin{figure}[!ht]
\centering
\includegraphics[width=0.32\textwidth,trim=7mm 2mm 22mm 8mm, clip]
{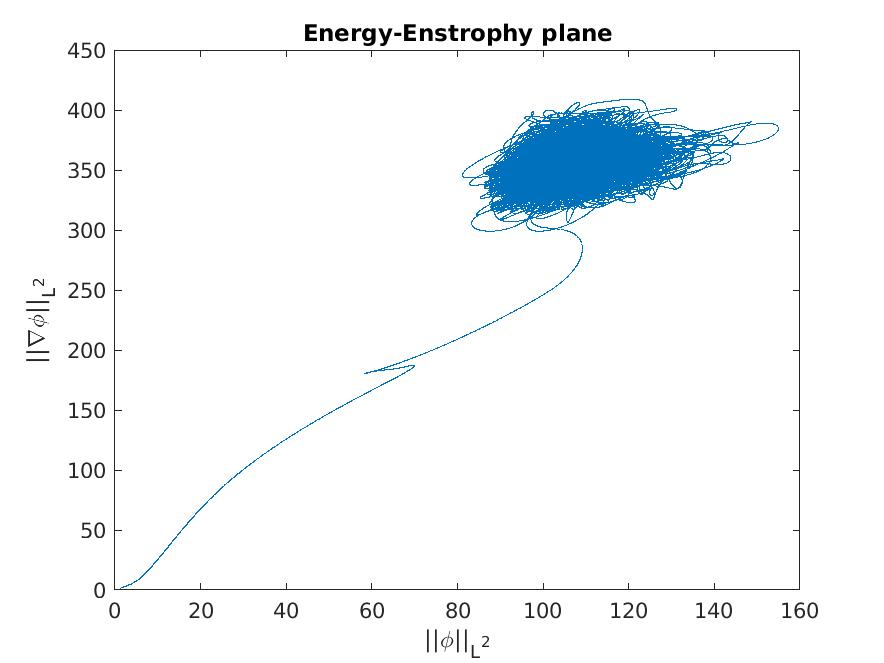}
\includegraphics[width=0.32\textwidth,trim=5mm 2mm 22mm 8mm, clip]
{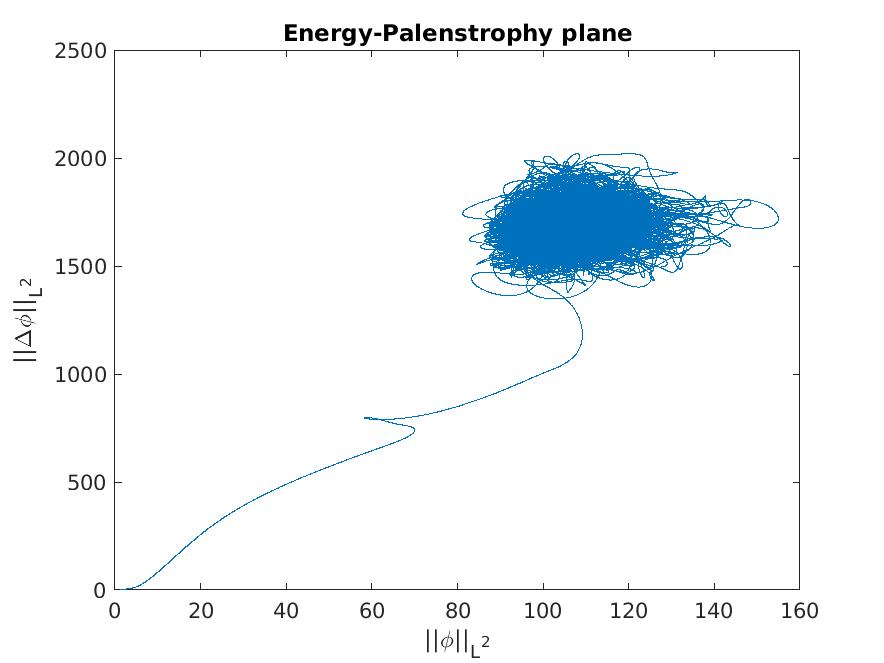}
\includegraphics[width=0.32\textwidth,trim=5mm 2mm 23mm 8mm, clip]
{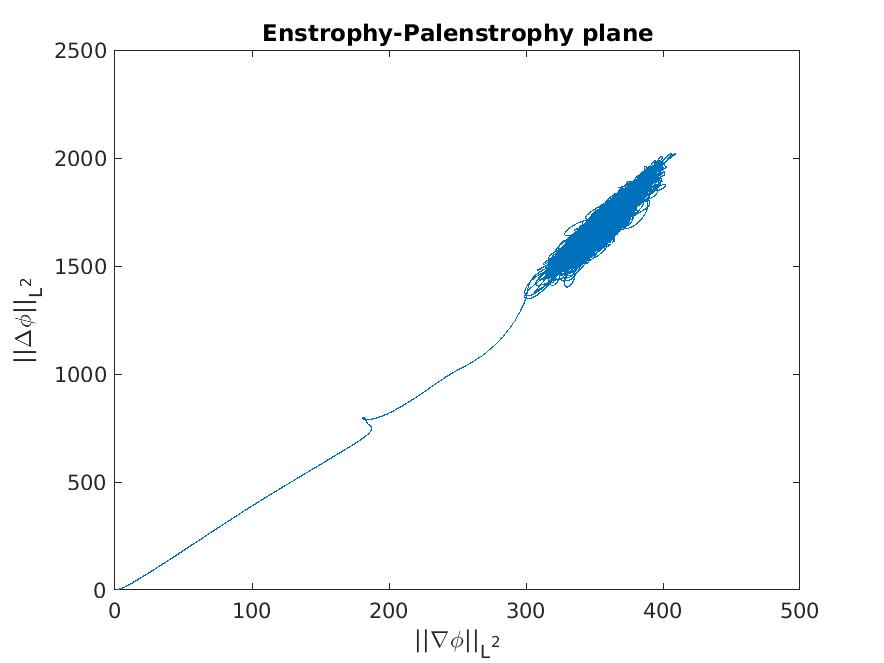}
\caption{\label{EnEnsPal} \scriptsize Plots of combinations of ``Energy'' $\|\vphi\|_L^2$, ``Enstrophy'' $\|\nabla\vphi\|_L^2$ and ``Palenstrophy'' $\|\Delta\vphi\|_L^2$ for times $0\leq t\leq3$.  At $t=0$, the solutions are close to the origin, and move roughly northeast until reaching a region in the northeast where they tend to remain, moving about chaotically.}
\end{figure}

We also display some quantities that occur in standard energy estimates.  For example, in taking the inner product with $\phi$ in \eqref{KSE_scalar} as in \eqref{L2phi}, the nonlinear term $|\nabla\phi|^2/2$ (bottom left of Figure \ref{prettyGraphics}) yields the term $\frac12\int_\Omega\phi|\nabla\phi|^2\,dx$ (Figure \ref{LpNormsNL}).  If this term happened to be positive, or even if it had a positive integral over aribitrary time intervals of the form $[0,T]$, $T>0$ this would be sufficient to obtain a bound on the $L^2$ norm on finite time intervals, and hence a proof of global well-posedness of the KSE.  Our numerical tests in Figure \ref{LpNormsNL} (left) show that this term is typically positive (at least, for the simulations we performed), but can intermittantly become negative for small windows of time.  In Figure \ref{prettyGraphics} (bottom middle), the integrand of this term, $\phi|\nabla\phi|^2$ can also be seen at time $t=10.0$.  Here, we see that the integrand is mostly positive in space, but has regions where it is strongly negative.  A deeper understanding of these negative regions may be important to understanding the well-posedness of the KSE.  Note also that, strictly speaking, the positivity of  $(\phi,|\nabla\phi|^2) + \|\Delta\phi\|_{L^2}^2$ would be sufficient to bound the $L^2$ norm, as from \eqref{L2phi}, one has
\begin{align*}
 \frac{1}{2} \frac{d}{dt} \lVert \phi \rVert_{L^{2}}^{2}
 + \tfrac{1}{2} (\lvert\nabla \phi \rvert^{2},\phi)
 + \lVert \Delta \phi \rVert_{L^{2}}^{2}
 = 
 - \lambda (\Delta \phi,\phi)
 \leq 
 \tfrac12\|\Delta \phi\|_{L^2}^2 + \tfrac{\lambda^2}{2}\|\phi\|_{L^2}^2,
\end{align*}
so that
\begin{align*}
\frac{d}{dt} \lVert \phi \rVert_{L^{2}}^{2} 
+ (\lvert\nabla \phi \rvert^{2},\phi)
+ \lVert \Delta \phi \rVert_{L^{2}}^{2} 
 \leq 
\lambda^2\|\phi\|_{L^2}^2. 
\end{align*}
A nearly identical calculation yields
\begin{align*}
\frac{d}{dt} \lVert u \rVert_{L^{2}}^{2} 
+ 2(u\cdot\nabla u,u)
+ \lVert \Delta u \rVert_{L^{2}}^{2} 
 \leq 
\lambda^2\|u\|_{L^2}^2. 
\end{align*}
Thus, in Figure \ref{LpNormsNL} we also display $(\phi,|\nabla\phi|^2) + \|\Delta\phi\|_{L^2}^2$ and $2(u\cdot\nabla u,u)+ \lVert \Delta u \rVert_{L^{2}}^{2}$.  In the figure, we see that the terms coming quantities coming from the non-linear terms are mostly positive, but occasionally become negative.  However, after adding the norm of the Laplacian, the quantity remains positive for all time.  Of course, if one could \textit{prove} that such positivity holds for general initial data, one could prove global well-posedness.

\begin{figure}[!ht]
\centering
\includegraphics[width=0.49\textwidth,trim=28mm 8mm 25mm 4mm, clip]
{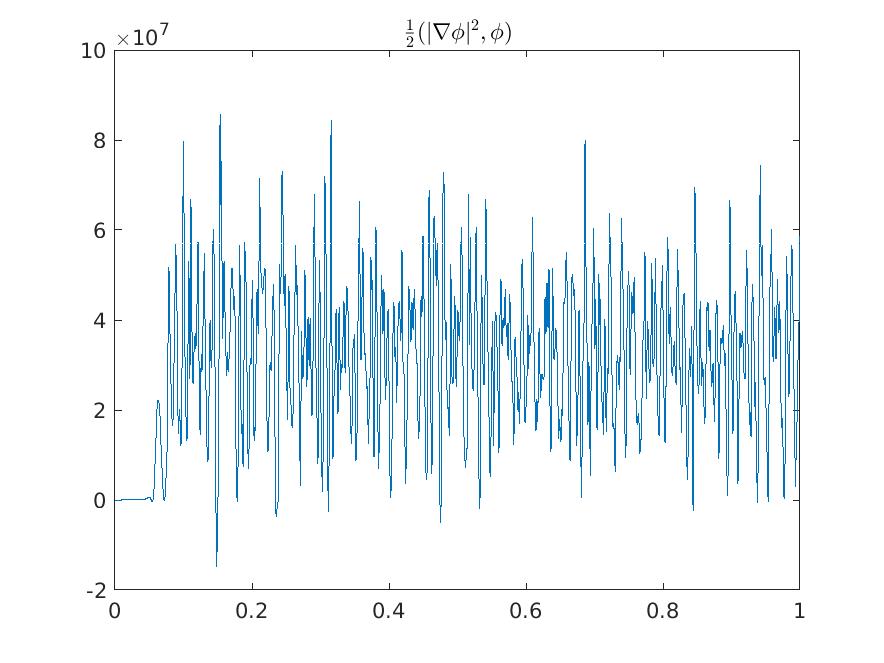}
\includegraphics[width=0.49\textwidth,trim=28mm 8mm 25mm 4mm, clip]
{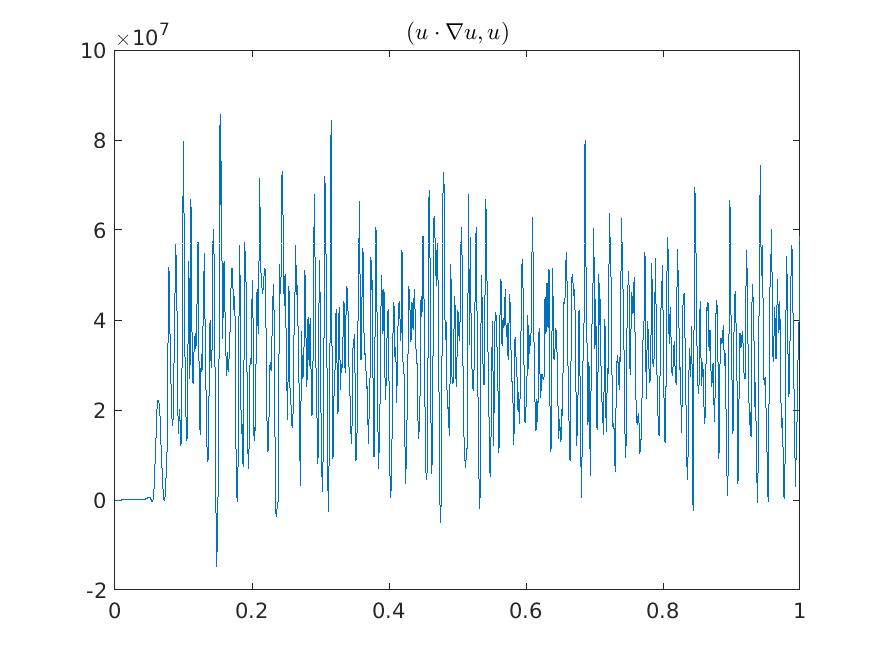}
\includegraphics[width=0.49\textwidth,trim=28mm 8mm 25mm 4mm, clip]
{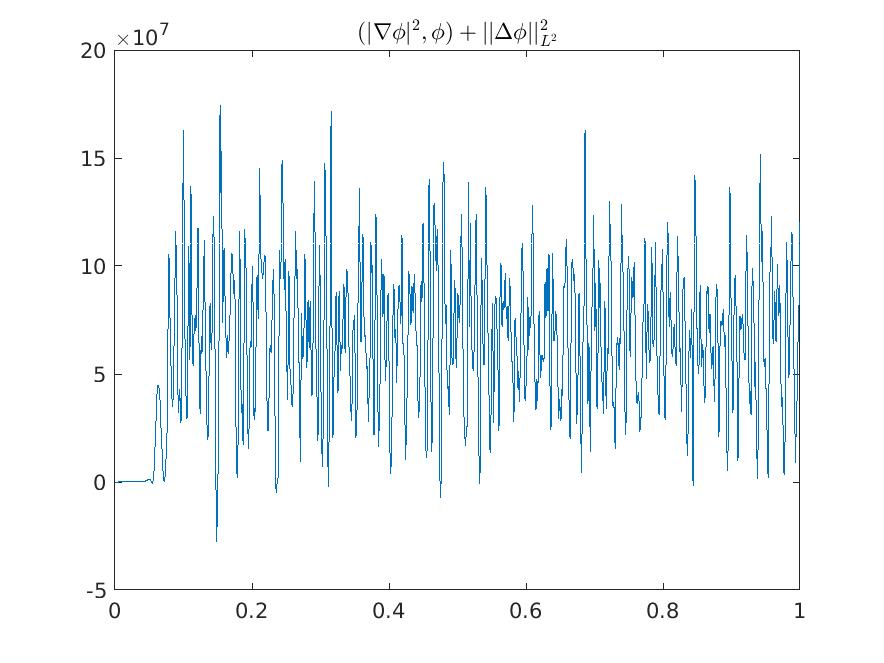}
\includegraphics[width=0.49\textwidth,trim=28mm 8mm 25mm 4mm, clip]
{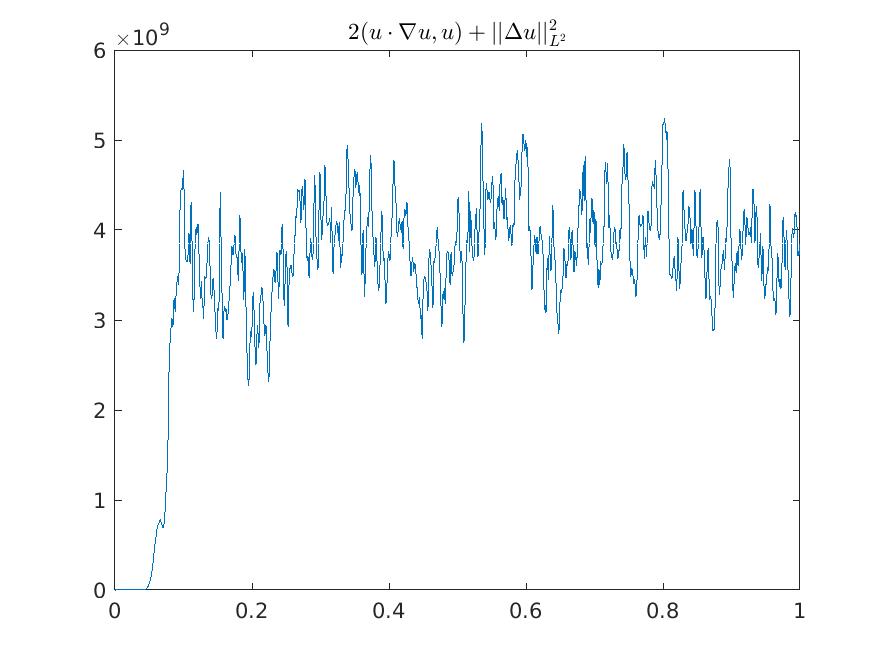}
\caption{\label{LpNormsNL} \scriptsize Quantities of interest in energy estimates: $\tfrac{1}{2} (\phi,\lvert\nabla \phi \rvert^{2})$ (top left), $(\phi,|\nabla\phi|^2)+\|\Delta\phi\|_{L^2}^2$ (top right),  $(u\cdot\nabla u,u)= -\frac12(\Delta\phi,|\nabla\phi|^2)$ (bottom left), and $2(u\cdot\nabla u,u)+\|\Delta u\|_{L^2}^2$ (bottom right). While the quantities in the left column have intermittent negativity, in our simulations, the quantities in the right column were positive for all $t\geq0$. (Only $0\leq t\leq1$ displayed (for visibility), but similar behavior was observed out to $t=10$. Norms were not normalized as in \eqref{normalizeNorms}.)} 
\end{figure}

Similar observations can be made at the level of the $u$ equation \eqref{KSE}.  In particular, regarding the cubic term in \eqref{star}, we notice
\begin{align}\label{ip_obs}
(u\cdot\nabla u,u)=\frac12\int_\Omega u\cdot\nabla|u|^2\,dx =-\frac12\int_\Omega (\nabla\cdot u)|u|^2\,dx = -\frac12\int_\Omega\Delta \phi|\nabla\phi|^2\,dx.
\end{align}
Note that, according to  \eqref{star}, positivity of this term would immediately imply $u\in L^\infty(0,T;L^2)$, and hence global existence for the KSE.
Thus, we also display $\Delta \phi|\nabla\phi|^2/2$ in Figure \ref{prettyGraphics} (bottom right), where we make the perhaps unsurprising observation that integrand in \eqref{ip_obs} is not of a single sign; however, it is surprising that integrand evidently has positive values that are far greater than the magnitude of the negative values.  Moreover, the regions where the integrand is negative seem to be concentrated in small, relatively thin pockets, while the positive regions dominate the domain.  Indeed, in simulations  negative values of $(u\cdot\nabla u,u)$ were observed only for short intervals of time (no larger than roughly $[t,t+0.003]$), but after a short time (roughly $t>0.08$), $(u\cdot\nabla u,u)$ remained positive.  This phenomenon was observed repeatedly in simulations, indicating that one possible route to proving global well-posedness of \eqref{KSE} might be to show that the negative part of $\Delta \phi|\nabla\phi|^2/2$ remain sufficiently small in comparison to its positive part.

\begin{figure}[!ht]
\centering
\includegraphics[width=0.32\textwidth,trim=48mm 30mm 36mm 5mm, clip]
{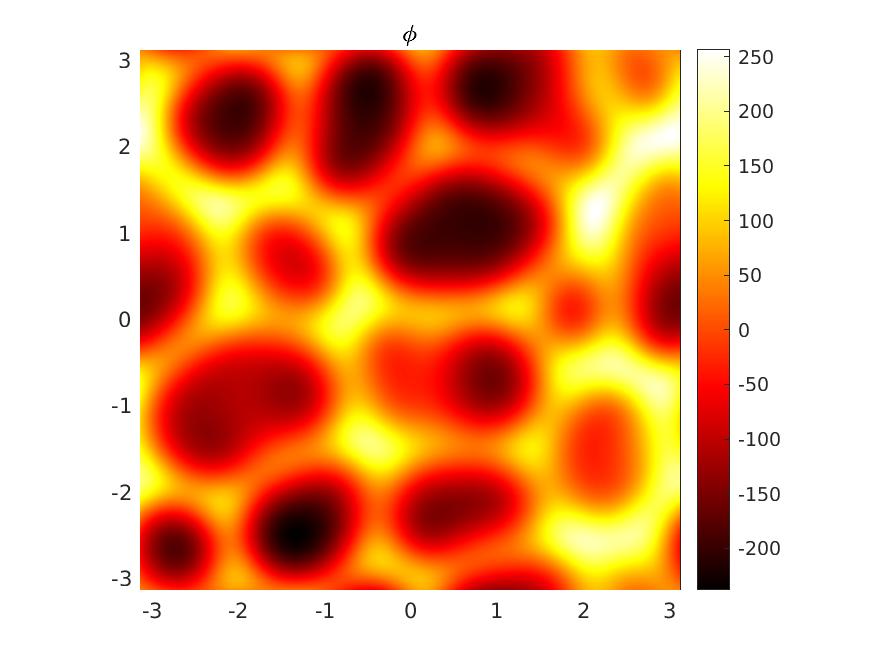}
\includegraphics[width=0.32\textwidth,trim=48mm 30mm 36mm 5mm, clip]
{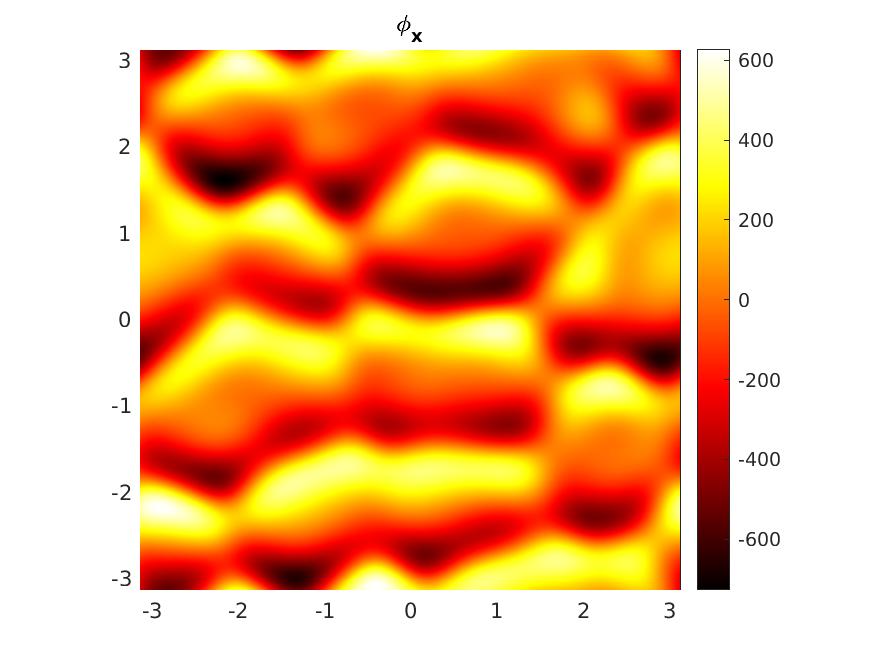}
\includegraphics[width=0.32\textwidth,trim=48mm 31mm 34mm 5mm, clip]
{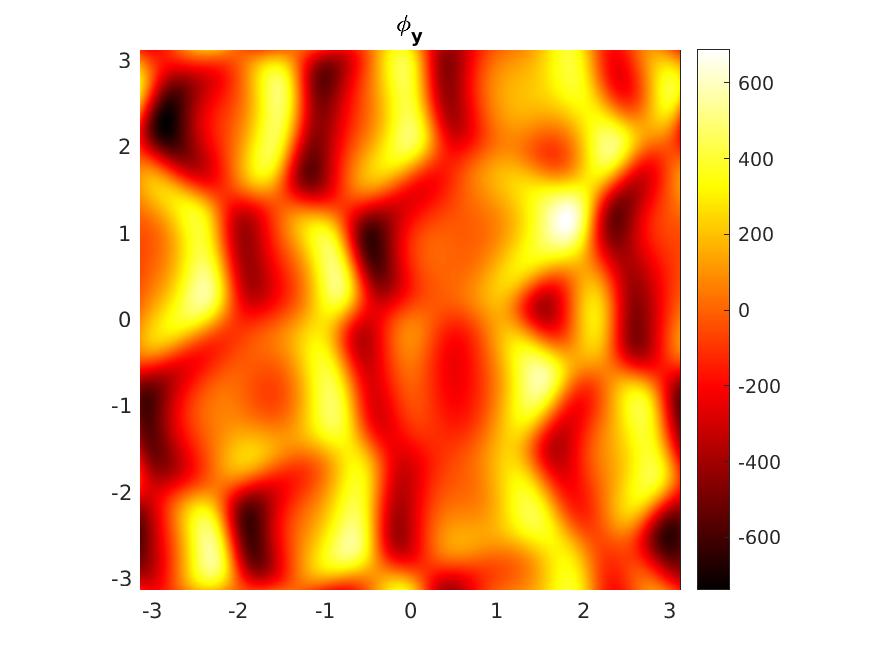}
\includegraphics[width=0.32\textwidth,trim=48mm 30mm 30mm 5mm, clip]
{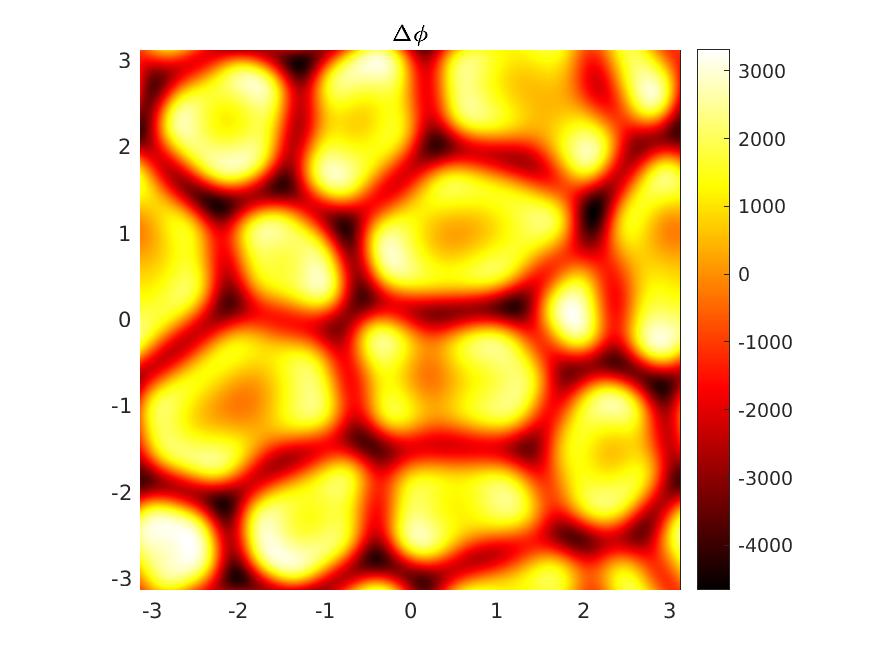}
\includegraphics[width=0.32\textwidth,trim=48mm 30mm 36mm 5mm, clip]
{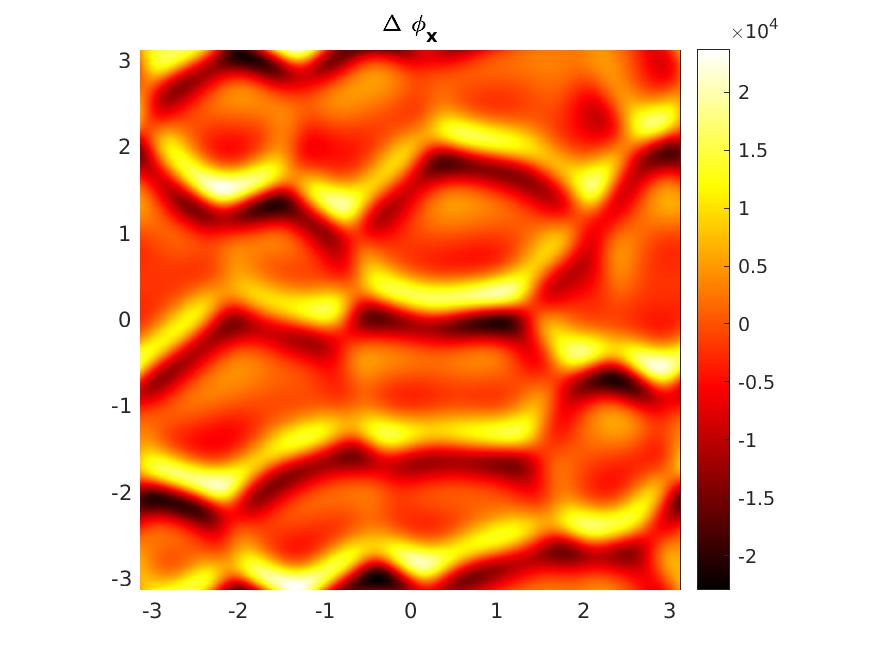}
\includegraphics[width=0.32\textwidth,trim=48mm 30mm 36mm 5mm, clip]
{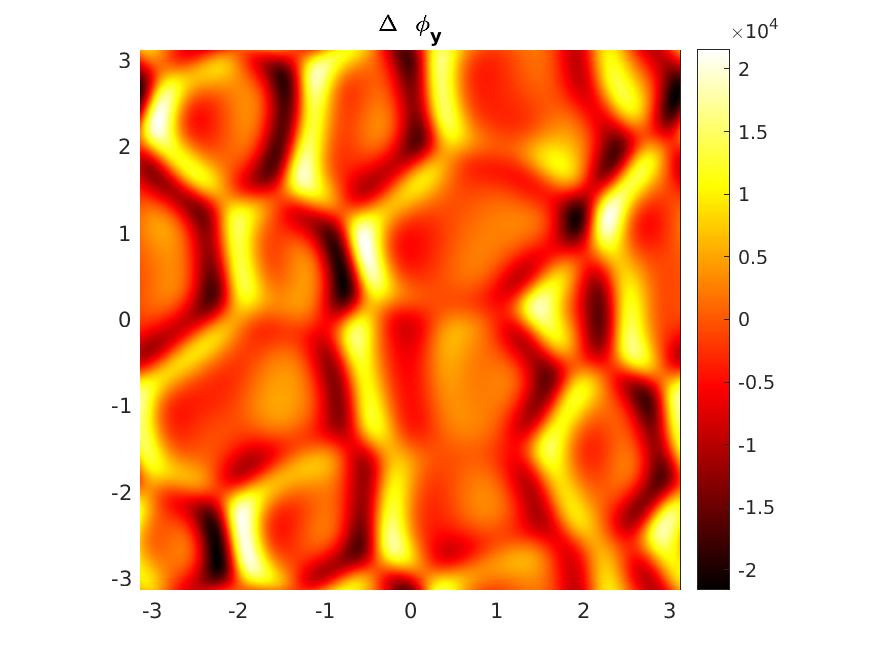}
\includegraphics[width=0.32\textwidth,trim=48mm 30mm 36mm 5mm, clip]
{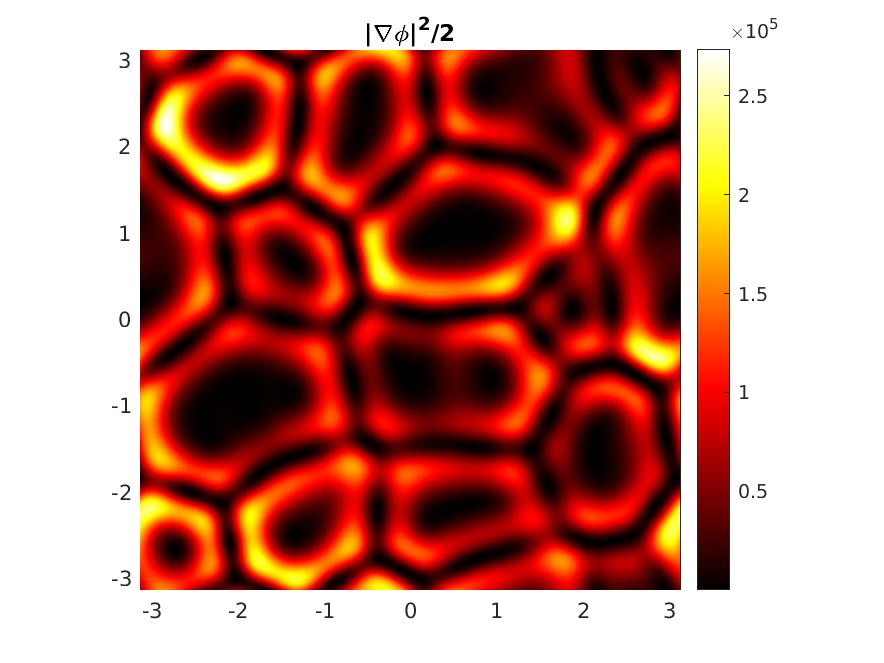}
\includegraphics[width=0.32\textwidth,trim=48mm 30mm 36mm 5mm, clip]
{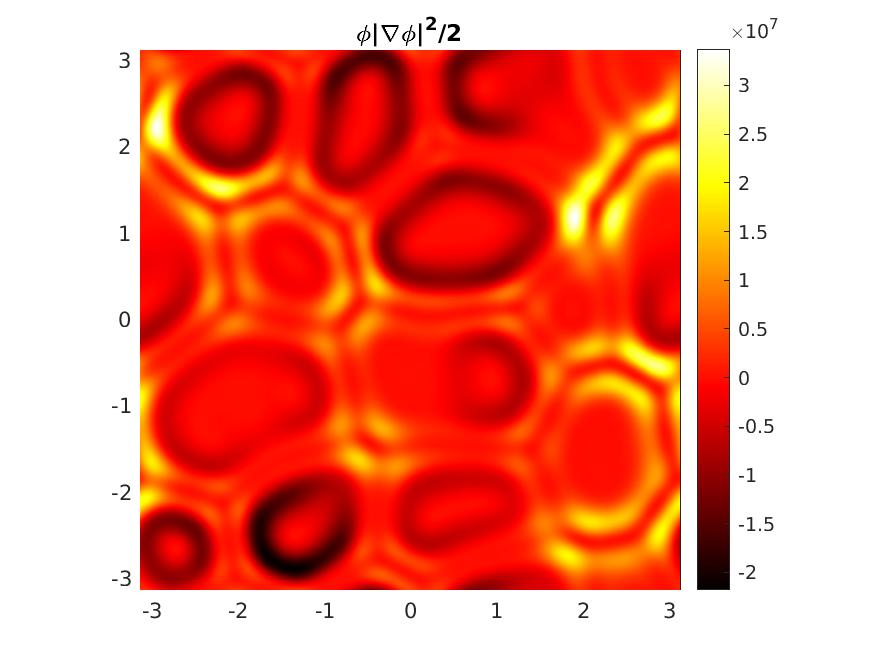}
\includegraphics[width=0.32\textwidth,trim=48mm 30mm 34mm 5mm, clip]
{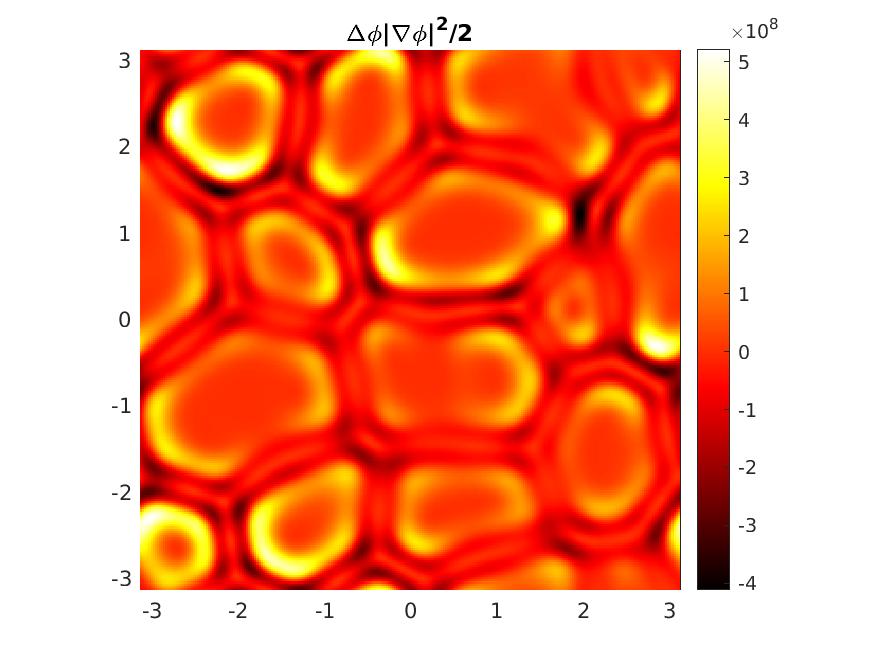}
\caption{\label{prettyGraphics} \scriptsize Solution $\phi$ to \eqref{KSE_scalar} at time $t=10.0$, and various quantities derived from $\phi$.  Here, $\Delta=\partial_{xx}+\partial_{yy}$.  Note that in terms of $u\triangleq\nabla\phi$, one has $\text{div}(u)=\Delta\phi$.}
\end{figure}

\FloatBarrier

\section{Appendix}

Here, we leave details concerning the local well-posedness of the KSE \eqref{KSE} in $H^{1}(\mathbb{T}^{N})$ for $N \in \{2,3\}$, formally stated in Theorem \ref{Theorem 2.2}. The proof follows the argument in \cite{Larios_Yamazaki_2020_rKSE}, which in turn follows \cite{Majda_Bertozzi_2002}; we only sketch the main steps. We consider a Galerkin approximation with $P_{n}$ being the projection onto the Fourier modes of order up to $n \in \mathbb{N} \cup \{0\}$: 
\begin{equation*}
P_{n}u(x) \triangleq  \sum_{\lvert k \rvert \leq n} \hat{u}(k) e^{ix \cdot k}. 
\end{equation*}
We consider the following Galerkin approximation system
\begin{subequations}
\begin{align}
& \partial_{t} u^{n} + P_{n} ((u^{n} \cdot \nabla) u^{n}) = -\lambda \Delta u^{n} - \Delta^{2} u^{n},  \label{3.1a}\\
& u^{n}(\cdot, 0) \triangleq P_{n} u^{in}(\cdot).  \label{3.1b}
\end{align}
\end{subequations} 
Relying on \cite[Theorem 3.1]{Majda_Bertozzi_2002} we can deduce the following statement: 
\begin{proposition}\label{Proposition 3.1}
Given initial data $u^{in} \in H^{1}(\mathbb{T}^{N})$, there exists $T = T(\lVert u^{in} \rVert_{H^{1}}) > 0$ such that \eqref{3.1a}-\eqref{3.1b} has a solution 
\begin{equation}\label{estimate 17}
u^{n} \in L^{\infty} ([0,T]; H^{1}(\mathbb{T}^{N})) \cap L^{2}([0,T]; H^{3}(\mathbb{T}^{N})); 
\end{equation}
additionally, $\partial_{t} u^{n} \in L^{2}([0,T]; H^{-1}(\mathbb{T}^{N}))$. Moreover, these bounds are independent of $n$. Finally, if $T^{\ast}$ is the maximal existence time and $T^{\ast} < \infty$, then $\limsup_{t\to T^{\ast}} \lVert u^{n}(t) \rVert_{H^{1}} = + \infty$. 
\end{proposition} 

\begin{proof} 
By \cite[Theorem 3.1]{Majda_Bertozzi_2002}, it can be immediately shown that given $u^{in} \in H^{1}(\mathbb{T}^{N})$, there exists a unique solution $u^{n}\in C^{1} ([0, T_{n}), H^{1}(\mathbb{T}^{N}) \cap O^{M})$ for some $T_{n} > 0$ and $O^{M} \triangleq \{f \in H^{1}(\mathbb{T}^{N}): \lVert f \rVert_{H^{1}} \leq M \}$. Now we can take $L^{2}(\mathbb{T}^{N})$-inner products on \eqref{3.1a}  with $u$ and then with $-\Delta u^{n}$, sum the resulting equations and prove that there exists a constant $c \geq 0$ such that 
\begin{equation*}
\frac{d}{dt} \lVert u^{n} \rVert_{H^{1}} \leq c(1+ \lVert u^{n} \rVert_{H^{1}})^{3} 
\end{equation*} 
We can fix such a constant $c > 0$ and deduce that $H^{1}(\mathbb{T}^{N})$-norm does not blow up for all $t < T^{\ast} \triangleq \frac{1}{2c(1+ \lVert u^{in} \rVert_{H^{1}})^{2}}$. Hence, $T_{n} > T \triangleq \frac{T^{\ast}}{2}$ for all $n \in \mathbb{N}$ and $u^{n}$ has the regularity of \eqref{estimate 17}. This regularity leads to to $\int_{0}^{T} \lVert \partial_{t} u^{n} \rVert_{H^{-1}}^{2} d\tau \lesssim 1$ by \eqref{KSE} and a standard argument  (see \cite[Cor. 3.2]{Majda_Bertozzi_2002}) completes the proof of Proposition \ref{Proposition 3.1}.  
\end{proof}
From the regularity of the solution to the Galerkin approximation due to Proposition \ref{Proposition 3.1}, we can deduce the following convergence results by standard compactness lemma (e.g., \cite[Theorem 5]{Simon_1987}, \cite[Lem. 4]{Simon_1990}): weak$^{\ast}$ in $L^{\infty} (0, T; H^{1}(\mathbb{T}^{N}))$; weak in $L^{2}(0, T; H^{3}(\mathbb{T}^{N}))$; strong in $L^{2}(0, T; H^{s}(\mathbb{T}^{N}))$ for $s \in [2, 3)$ and  $C([0,T]; H^{s}(\mathbb{T}^{N}))$ for $s \in [0, 1)$. Thereafter, verifying that the limiting solution indeed solves the KSE and proving its uniqueness is standard.  We refer to, e.g.,  \cite{Larios_Yamazaki_2020_rKSE} for details. 

 \section*{Acknowledgments}
 \noindent
Author A.L. was partially supported by NSF grant 
 CMMI-1953346. We wish to thank the anonymous reviewers for valuable comments that have improved our manuscript significantly. 
 
\begin{scriptsize}

\end{scriptsize}

\end{document}